\documentclass[oneside]{amsart}
  
\usepackage{amssymb}

\usepackage{graphicx}

\usepackage[cmtip,all]{xy}
\usepackage[latin1]{inputenc}
\usepackage{amssymb,amsmath}
\usepackage{verbatim}
\usepackage{color}
\usepackage{enumerate}
\usepackage{url}
\usepackage{geometry}
\usepackage{hyperref}

\usepackage{array}

\newtheorem{theorem}{Theorem}[section]
\newtheorem{lemma}[theorem]{Lemma}
\newtheorem{proposition}[theorem]{Proposition}
\newtheorem{conjecture}[theorem]{Conjecture}

\theoremstyle{definition}
\newtheorem{definition}[theorem]{Definition}

\newtheorem{assumption}[theorem]{Assumption}

\theoremstyle{remark}
\newtheorem{remark}[theorem]{Remark}

\numberwithin{equation}{section}

\newcommand\Z{\ensuremath{\mathbb Z}}\newcommand\A{\ensuremath{\mathbb A}}
\newcommand\Q{\ensuremath{\mathbb Q}}
\newcommand\C{\ensuremath{\mathbb C}}\newcommand\F{\ensuremath{\mathbb F}}

\newcommand\cM{\ensuremath{\mathcal M}}

\newcommand\Aut{\operatorname{Aut}}

\newcommand\End{\operatorname{End}}

\newcommand\Gal{\operatorname{Gal}}
\newcommand\GL{\operatorname{GL}}
\newcommand\Hom{\operatorname{Hom}}
\newcommand\Id{\operatorname{Id}}

\newcommand\ord{\operatorname{ord}}

\newcommand\Pic{\operatorname{Pic}}

\newcommand\Res{\operatorname{Res}}

\newcommand\CH{\operatorname{CH}}

\newcommand{\fp}{{\mathfrak{p}}}

\newcommand{\fc}{{\mathfrak{c}}}

\newcommand{\fa}{{\mathfrak{a}}}
\newcommand{\fq}{{\mathfrak{q}}}

\newcommand{\cE}{{\mathcal{E}}}

\renewcommand{\L}{\mathcal{L}}
\newcommand{\tto}[1]{%
\ifthenelse{\equal{#1}{}}{\to}{\stackrel{#1}{\to}}}

\newcommand{\Lp}{\L_p^g(\breve {\bf f}, \breve {\bf g}, \breve {\bf h})}

\def\cO{{\mathcal O}}

\newcommand{\mtx}[4]{\left(\begin{matrix}#1&#2\\#3&#4\end{matrix}\right)}

\def\p{\mathfrak p}
\newcommand{\comp}{\begin{picture}(6,5)(-3,-2)\put(0,1){\circle{2}} \end{picture}}\def\circ{\comp}

\def\cN{\mathcal N}

\newcommand{\ra}{\rightarrow}

\newcommand{\lra}{\longrightarrow}

\newcommand{\Qbar}{\overline\Q}

\def\XXint#1#2#3{{\setbox0=\hbox{$#1{#2#3}{\int}$}
\vcenter{\hbox{$#2#3$}}\kern-.5\wd0}}

\newcommand{\bff}{{\bf f}}
\newcommand{\bfg}{{\bf g}}
\newcommand{\bfh}{{\bf h}}
\newcommand{\ubff}{\breve{\bf f}}
\newcommand{\ubfg}{\breve{\bf g}}
\newcommand{\ubfh}{\breve{\bf h}}

\newcommand\cW{\ensuremath{\mathcal W}}
\newcommand\coh{\operatorname{H}}
\newcommand\AJ{\operatorname{AJ}}
\newcommand\fil{\operatorname{Fil}}
\newcommand{\Fil}{\fil}
\newcommand\norm{\operatorname{N}}
\newcommand{\etale}{\mathrm{et}}

\newcommand{\derham}{\mathrm{dR}}
\newcommand{\unitroot}{\mathrm{u-r}}

\newcommand{\fe}{\mathfrak{e}}
\newcommand{\ff}{\mathfrak{f}}
\newcommand{\dR}{\derham}

\renewcommand{\hom}{\mathrm{Hom}}

\begin{document}

\title{On the elliptic Stark conjecture in higher weight}

\author{Francesca Gatti}
\address{Departament de Matem\`atiques \\
Universitat Polit\`ecnica de Catalunya \\
Catalonia}
\email{francesca.gatti@upc.edu}

\author{Xavier Guitart}
\address{Departament de Matem\`atiques i Inform\`atica\\
Universitat de Barcelona\\
Catalonia\\
}
\email{xevi.guitart@gmail.com}


\date{\today}

\dedicatory{}

\begin{abstract}
We study the special values of the triple product $p$-adic $L$-function constructed by Darmon and Rotger at all classical points outside the region of interpolation. We propose conjectural formulas for these values that can be seen as extending the Elliptic Stark Conjecture, and we provide theoretical evidence for them by proving some particular cases.
\end{abstract}
\maketitle 


 \section{Introduction}\label{sec: intro}

 A common characteristic of $p$-adic $L$-functions is that they can be defined by interpolating critical special values of complex $L$-functions. Since complex $L$-functions are attached (at least conjecturally) to motives, this process can be seen in many instances as associating a $p$-adic $L$-function to a $p$-adic family of motives. In these cases, if the domain of the $p$-adic $L$-function is a $p$-adic space $\cW$ there is a family of motives ${\bf M}=\{M_x\}_{x\in\Sigma^{\mathrm{cl}}}$  parametrized by a subset $\Sigma^{\mathrm{cl}}\subset\cW$ of {\em classical} points which is dense  with respect to the Zariski topology. The region of  interpolation is  a  subset $\Sigma^{\mathrm{int}}\subset \Sigma^{\mathrm{cl}}$, again dense within $\cW$, with the property that for each $x\in\Sigma^{\mathrm{int}}$ there is a canonical period $\Omega_x\in\C$ in the sense of Deligne such that the value of the complex $L$-function $L(M_x,c_x)/\Omega_x$ at its critical point $c_x$ is algebraic; when multiplied by an appropriate Euler $p$-factor $\cE(M_x)\in\Qbar$, these values can be $p$-adically interpolated to a rigid-analytic function on $\cW$. The $p$-adic $L$-function attached to ${\bf M}$ is then a function $\L_p({\bf M},s)\colon \cW \ra \C_p$ such that 
\begin{align*}
  \L_p({\bf M},x)=\frac{\cE(M_x)}{\Omega_x}L(M_x,c_x), \quad x\in \Sigma^{\mathrm{int}}.
\end{align*}

In this framework, it is usually of great interest to study the values $\L_p({\bf M},x)$ at classical points $x\in\Sigma^{\mathrm{cl}}\setminus \Sigma^{\mathrm{int}}$ lying outside the region of interpolation, as it is believed they should encode a $p$-adic invariant associated to the idoneous motivic cohomology group of the motive $M_x$. 

A prototypical and classical example of this situation is Leopoldt's $p$-adic formula for the value at $s=1$ of the Kubota--Leopoldt $p$-adic $L$-function  associated to an even Dirichlet character $\chi$. The interpolation formula for this $p$-adic $L$-function $\L_p(\chi,s)$ reads
\begin{align*}
  \L_p(\chi,k)=(1-\chi(p)\omega(p)^{-k}p^{-k})L(\chi\omega^{-k},k) \text{ for }k\in\Z_{\leq 0},
\end{align*}
where $\omega\colon\Z/p\Z\ra \C^\times$ is the Teichmuller character. Therefore, $\L_p(\chi,s)$ for $s\in\Z_p$ can be interpreted as the $p$-adic $L$-function associated to $\{\Z(\chi)(k)\}_{k\in\Z}$, the family of Tate twists of the Dirichlet motive $\Z(\chi)$, with region of interpolation $\Sigma^{\mathrm{int}} = \Z_{\leq 0}$. The value at $k=1$ is then outside the region of interpolation, and Leopoldt's formula relates $\L_p(\chi,1)$ to the $p$-adic logarithm of a circular unit in the cyclotomic field $\Q(\zeta_N)$:
\begin{align*}
  \L_p(\chi,1)=-\frac{(1-\chi(p)p^{-1})}{\mathfrak{g}(\chi^{-1})}\sum_{j=1}^{N-1}\chi^{-1}(j)\log_p(1-\zeta_N^j),
\end{align*}
where $N$ is the conductor of $\chi$,  $\zeta_N$ is a $N$-th root of unity, and $\mathfrak{g}$ denotes the Gauss sum.

There are many other illustrative examples of this philosophy. Some classical and relatively recent formulas exhibiting this phenomenon are summarized in the survey \cite{tale}, but very recently there have been exciting developments in this direction, including \cite{BSV}, \cite{DR3}, \cite{Fo}, \cite{LSZ1}, \cite{LSZ2}.

In spite of that, all these scattered formulas in the literature do not provide a systematic and throughout study of the  collection of special values of a  $p$-adic $L$-function as a whole and it is not always easy to have a good understanding of the complete picture. The main aim of the present article is coming to terms with this problem, providing a complete, systematic (and often conjectural) answer to this question in the case of the Garret--Hida $p$-adic $L$-functions $$\L_p^f(\breve{\bf{f}}, \breve{\bf{g}}, \breve{\bf{h}}),\,  \L_p^g(\breve{\bf{f}}, \breve{\bf{g}}, \breve{\bf{h}}), \, \L_p^h(\breve{\bf{f}}, \breve{\bf{g}}, \breve{\bf{h}})$$ attached to a triple of (test vectors associated to) Hida families $\breve{\bf{f}}, \breve{\bf{g}}, \breve{\bf{h}}$ introduced by Darmon--Rotger \cite{DR1}. 

By symmetry, it is enough to consider one of these functions, say $\L_p^g(\breve{\bf{f}}, \breve{\bf{g}}, \breve{\bf{h}})$, which interpolates the square-roots of the central values of the classical $L$-function $L(\breve{f}_k\otimes \breve{g}_\ell\otimes \breve{h}_m,s)$ attached to the specializations of the Hida families at classical points of weights $k,\ell,m$  with $k,\ell,m\geq 2$ and $\ell\geq k+m$. 

There are currently some results and conjectures for the value of $\L_p^g(\breve{\bf{f}}, \breve{\bf{g}}, \breve{\bf{h}})(k,\ell,m)$  at {\em some} classical points $(k,\ell,m) \in  \Sigma^{\mathrm{cl}}\setminus \Sigma^{\mathrm{int}}$ with $k\geq 2$ and $\ell,m\geq 1$ that lie outside this region of interpolation. After reviewing them for the convenience of the reader, our goal is to complete the picture by formulating a conjectural formula for {\em each} point in $ \Sigma^{\mathrm{cl}}\setminus \Sigma^{\mathrm{int}}$, as well as proving evidence for this conjecture by proving some particular cases.  

In order to explain this idea more precisely and to review the known results and conjectures to date, it is convenient to briefly recall some terminology related to Hida families.

 Let $p\geq 3$ be a prime and let $\Lambda=\Z_p[[1+p\Z_p]]$ be the Iwasawa algebra. A cuspidal Hida family $\bf f$ of tame level $N_f$ and Nebentype character $\chi_f$ can be regarded as a power series ${\bf f} =\sum a_n({\bf f})q^n\in \Lambda_{\bf f}[[q]]$, where $\Lambda_{\bf f}$ is a finite flat extension of $\Lambda$. A point $\nu$ in the weight space $\cW_{\bf f}=\Hom(\Lambda_{\bf f},\C_p)$ is called classical crystalline of weight $k$ if its restriction to $\Lambda$ is of the form $x\mapsto \omega^k(x)x^k$ for some $k\in \Z_{\geq 2}$, where $\omega$ denotes the Teichmuller character. The specialization ${\bf f}_\nu=\sum \nu(a_n({\bf f}))q^n$ at such points is then a $p$-ordinary cuspidal eigenform of weight $k$, level $pN_f$ and character $\chi_f$. If ${\bf f}_\nu$ is old at $p$, which is {\em always} the case if $k>2$,  then it is the ordinary $p$-stabilization of an eigenform of level $N_f$ that we denote by $f_k$; if it is new at $p$, which can only occur if $k=2$, then we simply put $f_k = {\bf f}_\nu$. We will also be interested in some weight one specializations of cuspidal Hida families, although in this case such specializations are not guaranteed to be neither classical nor cuspidal. If $N$ is a multiple of $N_f$, a test vector for ${\bf f}$ of tame level $N$ is a family of the form $\sum\lambda_d {\bf f}(q^d)\in \Lambda_{\bf f}[[q]]$ with $\lambda_d\in \Lambda_{\bf f}$, where $d$ runs over the divisors of $N/N_f$.

 Let now $\bf{f}, \bf{g}, \bf{h}$ be $\Lambda$-adic cuspidal Hida families of tame levels $N_f$, $N_g$, and $N_h$,  and tame Nebentype characters $\chi_f$, $\chi_g$, $\chi_h$ satisfying that $\gcd(N_f,N_g,N_h)$ is squarefree and $\chi_f\chi_g\chi_h=1$. Put  $N=\operatorname{lcm}(N_f,N_g,N_h)$ and suppose that $p\nmid N$. If $\breve{\bf f}$, $\breve{\bf g}$, and $\breve{\bf h}$ are test vectors of tame level $N$ associated to ${\bf f}$, ${\bf g}$, and ${\bf h}$, one can consider the three variable $p$-adic $L$-function
 \begin{align*}
   \L_p^g(\breve{\bf f}, \breve{\bf g}, \breve{\bf h})\colon \cW_{\bf f}\times \cW_{\bf g}\times \cW_{\bf h}\lra \C_p
 \end{align*}
 constructed in  \cite{DR1}. 

For simplicity of notation and exposition (and also because this is the most interesting setting), let us assume that  ${\bf g}$ and ${\bf h}$ specialize to a classical modular form at all crystalline points of weight one. Then the set of classical  crystalline specializations of $ \L_p^g(\breve{\bf f}, \breve{\bf g}, \breve{\bf h})(k,\ell,m)$ can be divided into four regions as follows:
 \begin{enumerate}
 \item $\Sigma^{f}=\{(k,\ell,m)\colon k\geq 2,\ \ell,m\geq 1\text{ and } k\geq \ell + m\}$;
 \item $\Sigma^{g}=\{(k,\ell,m)\colon k\geq 2,\ \ell,m\geq 1\text{ and } \ell \geq k + m\}$;
 \item $\Sigma^{h}=\{(k,\ell,m)\colon k\geq 2,\ \ell,m\geq 1\text{ and } m \geq k + \ell\}$;
    \item $\Sigma^{\mathrm{bal}}= \{(k,\ell,m)\colon k\geq 2,\ \ell,m\geq 1\}\setminus \left(\Sigma^f\cup \Sigma^g\cup \Sigma^h \right)$.
    \end{enumerate}
The type of arithmetic information encoded by $ \L_p^g(\breve{\bf f}, \breve{\bf g}, \breve{\bf h})(k,\ell,m)$ depends on the region where $(k,\ell,m)$ lies. For example, $\Sigma^{g}$ is the region of classical interpolation and therefore $ \L_p^g(\breve{\bf f}, \breve{\bf g}, \breve{\bf h})(k,\ell,m)$ can be expressed in terms of the algebraic part of the central value of the classical $L$-function $L(\breve{f}_k,\breve{g}_\ell,\breve{h}_m,s)$ by the interpolation formula of \cite[Theorem 4.7]{DR1}. More recently, Hsieh \cite{Hsieh} found an explicit choice of test vectors for which an improved interpolation formula holds (see \cite[Theorem A]{Hsieh} or Theorem \ref{thm: Hsieh interpolation formula} below for the precise formula).

On the other hand, if $(k,\ell,m)\in \Sigma^{\mathrm{bal}}$, then $ \L_p^g(\breve{\bf f}, \breve{\bf g}, \breve{\bf h})(k,\ell,m)$ can be expressed in terms of the syntomic Abel--Jacobi image of a generalized diagonal cycle $\Delta_{k,\ell,m}$ in the product of Kuga--Sato varieties $W=W_{k-2}\times W_{\ell-2}\times W_{m-2}$ (here $W_{k-2}$ denotes the desingularization of the $(k-2)$-fold fiber product of the universal elliptic curve over the modular curve $X_1(N_f)$, and similarly for $W_{\ell-2}$ and $W_{m-2}$).  More precisely, \cite[Theorem 5.1]{DR1} states that
\begin{align*}
   \L_p^g(\breve{\bf f}, \breve{\bf g}, \breve{\bf h})(k,\ell,m) =(-1)^{(\ell-k-m+2)/2} \frac{\cE(f,g,h)}{\cE_0(g)\cE_1(g)}\mathrm{AJ}_p(\Delta_{k,\ell,m})(\omega_f\otimes \eta_g^{\mathrm{u-r}}\otimes \omega_h),
 \end{align*}
 where $\mathrm{AJ}_p$ is the syntomic Abel--Jacobi map on the Chow group of $W$, $\cE(f,g,h), \cE_0(g), \cE_1(g)$ are explicit Euler factors, and $\omega_f\otimes \eta_g^{\mathrm{u-r}}\otimes \omega_h\in H^{k+\ell+m-3}_{\mathrm{dR}}(W/\Q_p)$ is a certain cohomology class naturally attached to the forms $f$, $g$, and $h$.

 The cases $\Sigma^f$ and $\Sigma^h$ turn out to be symmetric, so it remains to consider $\Sigma^f$. In this case, the article \cite{DLR} can be viewed as the first step towards understanding the values of  $\L_p^g(\breve{\bf f}, \breve{\bf g}, \breve{\bf h})$ at classical weights in $\Sigma^f$ by means of the so-called Elliptic Stark Conjecture, which gives a conjectural formula for $\L_p^g(\breve{\bf f}, \breve{\bf g}, \breve{\bf h})(2,1,1)$ under an additional classicality assumption on the weight one specialization of ${\bf g}$. The aim of the present article is to extend the conjectural picture proposed in \cite{DLR} to all classical weights in $\Sigma^f$,  thus completing the (partially conjectural) understanding of $  \L_p^g(\breve{\bf f}, \breve{\bf g}, \breve{\bf h})(k,\ell,m)$ at all classical weights.

 A striking feature of the Elliptic Stark Conjecture of \cite{DLR} is that it gives insight into an arithmetic problem related to the Birch and Swinnerton--Dyer conjecture in rank $2$, as we next recall.  Let  $E$ be an elliptic curve defined over $\Q$ and denote by $V_p(E)$ its $p$-adic Tate module, viewed as a representation of  $G_\Q=\Gal(\Qbar/\Q)$. Let $$g\in M_1(N_g,\chi_g)_L \text{ and } h\in M_1(N_h,\chi_h)_L$$ be eigenforms of weight one, Fourier coefficients in a number field $L$, and whose Nebentypus characters satisfy that $\chi_g\cdot \chi_h=1$. Denote by $V_g$ (resp. $V_h$) the Artin representation over $L$ attached to $g$ (resp. $h$), and by $\rho_{gh}$ the representation associated to $V_{gh}=V_g\otimes V_h$. The $L$-function $L(E,\rho_{gh},s)$ associated to the Galois representation $V_p(E)\otimes V_{gh}$ coincides with $L(f\otimes g\otimes h,s)$, the Garret--Rankin--Selberg $L$-function attached to $f$, $g$, and $h$, where $f\in S_2(N_f)$ stands for the modular form of weight $2$ associated to $E$. 
 The equivariant refinement of the Birch and Swinnerton--Dyer (BSD) conjecture then predicts that
\begin{align*}
  \mathrm{ord}_{s=1}L(E,\rho_{gh},s) \stackrel{?}{=} \dim_L \Hom_{G_\Q}(V_{gh},E(H)\otimes L).
\end{align*}

Assume that all the local root numbers of $L(E,\rho_{gh},s)$ are $+1$, and that $\ord_{s=1}L(E,\rho_{gh},1)=2$. Suppose that $p\geq 3$ is a prime with $\ord_p(N_f)\leq 1$ and $p\nmid N_g\cdot N_h$, and denote by $N$ the prime-to-$p$ part of $\operatorname{lcm}(N_f,N_g,N_h)$. Denote by $\alpha_g,\beta_g$ the two roots of characteristic polynomial of the Hecke operator $T_p$ acting on $g$  and let $g_\alpha$ be the  $p$-stabilization of $g$ such that $U_p(g_\alpha) = \alpha_g g_\alpha$ (and define similarly $\alpha_h,\beta_h$ and $h_\alpha$). Let $ {\bf f},  {\bf g},  {\bf h}$ be  Hida families of tame levels $N_f$, $N_g$, $N_h$ and tame Nebentype characters $\chi_f=1$, $\chi_g$, and $\chi_h$ such that ${f}_2 = f,  {\bf g}_1 = g_\alpha, {\bf h}_1 = h_\alpha.$

Let us also assume the classicality hypothesis for $g_\alpha$ (labeled as hypothesis C in \cite{DLR}). 
A crucial ingredient in the conjecture is a regulator, defined roughly as follows. Using the action of the geometric Frobenius element $\sigma_p$ at $p$ on $V_{gh}$, one identifies a certain $2$-dimensional subspace $V_\alpha\subset V_{gh}$ on which $\sigma_p$ acts with eigenvalues $\alpha_g\alpha_h$ and $\alpha_g\beta_h$. Suppose now that $\Phi_1,\Phi_2$ is a basis for $\Hom_{G_\Q}(V_{gh},E(H)\otimes L)$, and denote by $v_1,v_2$ a basis of $V_{\alpha}$. The \emph{regulator} is then defined as
\begin{align*}R_{g_\alpha}(E,\rho_{gh})=\det \left(
  \begin{array}{cc}\log_{E,p}(\Phi_1(v_1)) & \log_{E,p}(\Phi_1(v_2))\\
\log_{E,p}(\Phi_2(v_2)) & \log_{E,p}(\Phi_2(v_1))
\end{array}
\right),
\end{align*}
where $\log_{E,p}\colon E(H)\otimes L\ra \C_p\otimes L$ is the $p$-adic formal group logarithm.

The Elliptic Stark Conjecture then states that there exists a choice of test vectors $\breve {\bf f}, \breve {\bf g}, \breve {\bf h}$ of tame level $N$ associated to $ {\bf f},  {\bf g},  {\bf h}$ such that

\begin{align}\label{eq: ESC}
\L_p^g(\breve {\bf f}, \breve {\bf g}, \breve {\bf h})(2,1,1) = \frac{R_{g_\alpha}(E,\rho_{gh})}{\log_p(u_{g_\alpha})},
\end{align}
where $\log_p\colon H^\times \otimes L \ra \C_p\otimes L$ is the usual $p$-adic logarithm and $u_{g_\alpha}\in \cO_H[1/p]^\times \otimes L$ is the so-called Gross--Stark unit attached to $g_\alpha$, defined in \cite[\S1.2]{DLR}. The main theoretical evidence supporting the Elliptic Stark Conjecture stems from \cite[Theorem 3.1]{DLR}, which proves it in the particular case where $g$ and $h$ are theta series of the same imaginary quadratic field in which $p$ splits. 

The aim of this article is to study natural generalizations of the Elliptic Stark Conjecture, as well as to provide theoretical evidence for them. The first setting we consider is the case where $f\in S_k(N_f,\chi_f)$ is a modular form of weight $k=r+2\geq 2$ and  possibly non-trivial character $\chi_f$, and $g$ and $h$ are modular forms of weight $1$ satisfying now that $\chi_f\cdot\chi_g\cdot \chi_h = 1$. That is to say, in \cite{DLR} the modular form $f$ is the weight two modular form attached to an elliptic curve $E/\Q$, and we now allow $f$ to have higher weight, non-trivial character, and non-rational Fourier coefficients. Observe that the condition $\chi_f\cdot\chi_g\cdot \chi_h = 1$ ensures that $V_f\otimes V_g\otimes V_h$ is Kummer self-dual, and we denote by $L(f,\rho_{gh},s)$ its $L$-series.

In this setting, the role played by the elliptic curve $E$ in the previous discussion is played by the motive attached to $f$, which arises from the Kuga--Sato variety $W_r$ by means of a suitable projector $e_f$, constructed using automorphisms of $W_r$ and Hecke operators,  which projects to the $f$-isotypical component.

Denote  by\footnote{Observe that the condition $\chi_f\cdot\chi_g\cdot\chi_h=1$ forces $k+\ell+m$ to be even, so in particular $k$ is even when $\ell=m=1$} $\CH^{k/2}(W_r/H)_0$ the Chow group of $H$-rational null-homologous cycles in $W_r$ of codimension $k/2$. An equivariant version of the Beilinson conjecture predicts that
\begin{align*}
  \ord_{s=k/2}L(f,\rho_{gh},s) \stackrel{?}{=} \dim_L\Hom_{G_\Q}(V_{gh},e_f\CH^{k/2}(W_r/H)_{0}\otimes L).
\end{align*}
In \S\ref{sec: higher weight esc} we generalize the Elliptic Stark Conjecture to this setting, in which $(f,g,h)$ are of weights $(k,1,1)$ with $k\geq 2$. For this, we extend the definition of Darmon--Rotger--Lauder's regulator, which will now involve the $p$-adic Abel--Jacobi map of cycles in $\CH^{k/2}(W_r)_0$ as a substitute for the $p$-adic logarithm of points on $E$. We then conjecture a formula akin to \eqref{eq: ESC}, namely an equality between $\L_p^g(\breve {\bf f}, \breve {\bf g}, \breve {\bf h})(k,1,1)$ and the regulator. 

In order to provide evidence for the conjecture, in \S \ref{sec: factorization formula} we prove it in a particular case where $g$ and $h$ are theta series of the same imaginary quadratic field in which $p$ splits. The structure of the proof follows the strategy devised in \cite[\S 3.2]{DLR}: we prove a factorization formula of the $p$-adic $L$-function in terms of a product of $p$-adic Rankin $L$-functions and a Katz $p$-adic $L$-function, and we invoke the main theorem of \cite{BDP1}. We remark that our factorization formula of $\Lp$ generalizes those of \cite{DLR} and \cite{CR}, and in addition we provide a simpler proof by taking advantage of the powerful $p$-adic triple $L$-function recently constructed by Hsieh \cite{Hsieh}. 

We also observe that the results of \cite{BDP1} play a key role in the proof of this particular case of the conjecture, for they allow to relate special values of $p$-adic Rankin $L$-functions with $p$-adic Abel--Jacobi images of Heegner cycles. But the main result of \cite{BDP1} (and the more general version of \cite[\S 4.1]{BDP2}) holds in the wider context of generalized Heegner cycles, and one might naturally wonder whether this is a manifestation of a more general version of the Elliptic Stark Conjecture in which $g$ and $h$ are of weight $\geq 2$. This is precisely the study that we undertake in \S \ref{sec: general unbalanced weights}, which as mentioned earlier is also motivated by the aim of providing a formula for $\Lp$ at all classical weights where $f$ is dominant.  

The second setting that we consider, to which we devote  \S \ref{sec: general unbalanced weights}, is that of modular forms $(f,g,h)$ of weights $(k,\ell,m)$ with $k\geq \ell+m$ and $l,m\geq 2$. As we will see in Conjecture \ref{conj: more general weights},  $\Lp(k,\ell,m)$ is expected to be related to a certain regulator (of a more geometric flavor in this case) of cycles on the motive attached to $f\otimes g\otimes h$. In order to provide some theoretical evidence for this conjecture, in \S \ref{subsection: The proof of a special case} we also prove it in a certain particular case where $g$ and $h$ are theta series of the same imaginary quadratic field in which $p$ splits.

Note that in order to complete the study of $\L_p^g(\breve {\bf f}, \breve {\bf g}, \breve {\bf h})$ at the region $\Sigma^f$ one should also consider weights of the form $(k,\ell,1)$ and $(k,1,\ell)$ with $k,\ell\geq 2$. As it will be apparent from the contents of sections \S \ref{sec: higher weight esc} and \S \ref{sec: general unbalanced weights}, this case is in fact a combination of the previous two cases and can be dealt with the same sort of techniques, so we do not include it in our analysis.  

We finally remark that in this note we study the values $\L_p^g(\breve{\bf f}, \breve{\bf g}, \breve{\bf h})(k,\ell,m)$ at classical points under the assumption that the classical $L$-function $L(\breve{f}_k\otimes \breve{g}_\ell\otimes \breve{h}_m,s)$ vanishes. The case where $L(\breve{f}_k\otimes \breve{g}_\ell\otimes \breve{h}_m,s)\neq 0$ will be investigated in the forthcoming work \cite{GGMR}.
 
\vspace{0.15cm}
{\bf Notations.} Throughout the article $p$ will denote an odd prime. We fix embeddings $\Qbar\hookrightarrow \C$ and $\Qbar\hookrightarrow \C_p$, where $\C_p$ denotes the completion of $\Qbar_p$. If $L\subset\Qbar$ is a number field we denote by $L_p$ the completion of $L$ in $\Qbar_p$ under this embedding.  If $L$ is a field we will denote by $S_k(N,\chi)_L$ the space of modular forms of level $k$ and Nebentypus $\chi$ with Fourier coefficients in $L$ (and when $L=\Qbar$ we will usually suppress it from the notation). If $V$ and $W$ are representations of a group $G$ over a field $L$, then $ W^V = \sum_{\phi \in\Hom_G(V,W)}\phi(V)$ denotes the $V$-isotypical component of $W$. If $\psi$ is a Hecke character of an imaginary quadratic field $K$, we will denote by $V_\psi$ the $2$-dimensional $G_\Q$-representation obtained by  induction.

\vspace{0.15cm}
{\bf Acknowledgments.} We are grateful to Victor Rotger for suggesting the problem to us and for his constant help during the preparation of this work. Gatti was partially supported by project MTM2015-63829-P, and Guitart was partially supported by projects MTM2015-66716-P and MTM2015-63829-P. This work has received funding from the European Research Council (ERC) under the European Union's Horizon 2020 research and innovation programme (grant agreement No 682152).

\section{The conjecture in weights $(k,1,1)$}\label{sec: higher weight esc}
The goal of this section is to formulate a generalization of the Elliptic Stark Conjecture for a triple of forms $(f,g,h)$ of weights $(k,1,1)$ with $k\geq 2$. We begin by recalling in \S\ref{subsec: triple L} the three variable $p$-adic $L$-function constructed in \cite{DLR}, which interpolates special values of Garret-Rankin's triple product $L$ function along Hida families, as well as the improved interpolation formulas arising from \cite{Hsieh}. In \S\ref{subsec: kuga sato} we briefly review the properties of Kuga--Sato varieties and $p$-adic Abel--Jacobi maps that we will need in order to define the regulator. Finally, in \S\ref{sec: the conjecture}, we construct the generalized regulator and we state the conjecture.
\subsection{The triple product $p$-adic $L$-function}\label{subsec: triple L}

	\subsubsection{Hida families}\label{subsec: hida families}
Let $\Gamma:=1+p\Z_p$ and let $\Lambda:=\Z_p[[\Gamma]]$ be the Iwasawa algebra. We denote by $\cW:=\mathrm{Spf}(\Lambda)$ the usual \textit{weight space},  which has the property that for any $p$-adic ring $A$ the set of $A$-valued points of $\cW$ is  given by $$\cW(A)=\hom_{\Z_p-\mathrm{alg}}(\Lambda,A)=\hom_\mathrm{cts}(\Gamma,A^\times).$$ 
One can attach a weight space to any finite flat extension $\Lambda_0$ of $\Lambda$ by defining $\cW_0:=\mathrm{Spf}(\Lambda_0)$. This space comes naturally equipped with a \textit{weight map} $\kappa:\cW_0\longrightarrow\cW,$ induced by the inclusion $\Lambda\subseteq \Lambda_0$. An element $\nu$ of $\cW(\C_p)$ is called \textit{classical} if it is of the form $\nu_{k,\epsilon}: x\mapsto \epsilon(x)x^k$ for some Dirichlet  character $\epsilon$ of conductor a power of $p$ and some $k\in \Z_{\geq 2}$. An element  $z\in\cW_0(\C_p)$ is called \textit{classical} if its restriction $\kappa(z)$ to $\Lambda$ is classical. A classical point $z\in\cW_0(\C_p)$ is called \textit{crystalline} if $\kappa(z)$ is of the form $\nu_{k,\omega^{k}}$, where $\omega:\Z_p^\times\rightarrow\mu_{p-1}$ denotes the Teichmuller character. In order to simplify the notation, we will write in this case $\kappa(z)=k$. We will denote by $\cW_0^\mathrm{cl}$ the set of classical points of $\cW_0$ and by $\cW_0^{\circ}$ the subset of crystalline points.

Let $N$ be a positive integer such that $p\nmid N$  and let $\chi:(\Z/N\Z)^\times\rightarrow\C_p^\times$ be a Dirichlet character.
\begin{definition}\label{def: hida family}
	A \textit{Hida family} of tame level $N$ and tame Nebentype character $\chi$ is a triple ${\bf f} =(\Lambda_{\bf f},\cW_{\bf f},{\bf f})$, where:
	\begin{enumerate}[\indent $i)$]
		\item $\Lambda_\bff$ is a finite flat extension of $\Lambda$;
		\item $\cW_\bff$ is a rigid analytic open subvariety of $\mathrm{Spf}(\Lambda_\bff)$;
		\item $\bff=\sum a_n(\bff)q^n\in\Lambda_\bff[[q]]$ is a formal series such that, for each $\nu\in\cW^\mathrm{cl}_{\bff}$, with $\kappa(\nu)=\nu_{k,\epsilon}$, the \textit{specialisation at $\nu$} $$\bff_\nu:=\sum_{n=1}^{\infty}\nu(a_n(\bff))q^n$$ is the $q$-expansion of a classical $p$-ordinary eigenform of weight $k$ and Nebentypus character $\chi\epsilon \omega^{-k}$. 
	\end{enumerate}
We will denote by $S_{\Lambda_\bff}^{\ord}(N,\chi)$ the set of such Hida families.
\end{definition}

Note that, if we restrict to $\cW_\bff^{\circ}$, then all the specializations of $\bff$ have  Nebentypus $\chi$. In particular, since $p$ does not divide the level of $\chi$, if $\nu$ has weight $k>2$, then  by \cite[Lemma 2.1.5]{How07}, the specialisation  $\bff_\nu$ is old at $p$; we will denote $f_\nu\in S_k(N,\chi)$ the newform whose $p$-stabilisation is $\bff_\nu$.   If $k=2$ then $\bff_\nu$ can be either old or new. In this case we denote  $f_\nu:=\bff_\nu$ if it is new, while, if $\bff_\nu$ is old at $p$, we denote $f_\nu$ the newform whose $p$-stabilisation is $\bff_\nu$.

\subsubsection{The complex Garrett--Rankin triple product $L$-function}
Let $$f\in S_k(N_f,\chi_f), \quad g\in S_\ell(N_g,\chi_g), \quad h\in S_m(N_h,\chi_h)$$
be three normalised newforms, cuspidal if they have weight $\geq2$, and assume that $\chi_f\cdot\chi_g\cdot\chi_h=1$. We denote by $V_f$, $V_g$, and $V_h$ the corresponding $2$-dimensional $p$-adic Galois representations. 

The \textit{Garrett--Rankin triple product $L$-function} $L(f\otimes g\otimes h,s)$ is the complex  $L$-function attached to the tensor product $V_{fgh}:=V_f\otimes V_g\otimes V_h$. 

It is defined by an Euler product which is absolutely convergent in the half plane $\mathrm{Re}(s)>\frac{k+\ell+m-1}{2}$. With the appropriate Euler factors at infinity, the completed function
$$\Lambda(f\otimes g\otimes h, s)=L_\infty(f\otimes g\otimes h,s)L(f\otimes g\otimes h,s)$$
extends to the whole complex plane and satisfies a functional equation of the form
\begin{equation}\label{eq: functional equation complex triple}
	\Lambda(f\otimes g\otimes h, s)=\epsilon(f,g,h)\Lambda(f\otimes g\otimes h, k+\ell+m-2-s),
\end{equation}
where $\epsilon(f,g,h)\in\{ \pm1 \}$ is the \textit{sign} of the functional equation. The center of symmetry with respect to (\ref{eq: functional equation complex triple}) is then $c:=\frac{k+\ell+m-2}{2}$, at which $L(f\otimes g\otimes h,s)$ has no pole.  Note that the condition $\chi_f\cdot\chi_g\cdot\chi_h=1$ implies that $k+\ell+m$ is even, so that $c\in\Z$, and moreover, $c$ is a \textit{critical point} for the $L$-function, meaning that $L_\infty(f\otimes g\otimes h,s)$ has no poles at $s=c$.
\begin{definition}
	A triple of weights $(k,\ell,m)\in\Z^3$ is called \textit{unbalanced} if one of the weights is greater or than or equal to the sum of the other two (in which case the greater weight is called \emph{dominant weight}). Otherwise, the triple is called \textit{balanced}.
\end{definition}	
The sign of the functional equation can be expressed as a product of \textit{local signs} over the places of $\Q$. More precisely, if $N:=\text{lcm}(N_f,N_g,N_h)$, then 
\begin{equation}\label{eq: sign of the functional equation}
	\epsilon(f,g,h)=\prod_{v\mid N\cdot\infty}\epsilon_v(f,g,h),
\end{equation}
and the local sign at infinity depends on whether the weights are balanced or unbalanced: 
$$\epsilon_\infty(f,g,h)=\begin{cases}
+1 & \iff (k,\ell,m) \text{ unbalanced} \\ -1 & \iff (k,\ell,m) \text{ balanced.}
\end{cases}$$
For more details in the study of the complex $L$-function, see \cite{PSR87}.
\subsubsection{The triple product $p$-adic $L$-function}
Let $\bff,\bfg,\bfh$ be three Hida families of tame levels $N_\bff$, $N_\bfg$, $N_\bfh$ and tame Nebentypus characters $\chi_\bff$, $ \chi_\bfg$, and $\chi_\bfh$  such that $\chi_\bff\cdot\chi_\bfg\cdot\chi_\bfh=1$. 
As in \cite[Hypothesis (sf) and (CR)]{Hsieh}, we assume the following hypothesis.
\begin{assumption}\label{ass: (sf) and (CR)} 
	\begin{enumerate}[\indent (1)]
		\item $\gcd(N_\bff,N_\bfg,N_\bfh)$ is square free;
		\item the residual representation $\bar{\rho}_\bfg:G_\Q\rightarrow\mathrm{GL}_2(\bar{\F}_p)$ is absolutely irreducible and $p$-distinguished (i.e., its semisimplification does not act as multiplication by scalars when restricted to a decomposition group at $p$).	\end{enumerate}
\end{assumption}
Define the set $\cW_{\bff\bfg\bfh}^{\circ}:=\cW_{\bff}^{\circ}\times\cW_{\bfg}^{\circ}\times\cW_{\bfh}^{\circ}$ of triples of classical crystalline points for $\bff,\bfg,\bfh$. It can be decomposed as $$\cW_{\bff\bfg\bfh}^{\circ}=\cW_{\bff\bfg\bfh}^f\sqcup\cW_{\bff\bfg\bfh}^g\sqcup\cW_{\bff\bfg\bfh}^h\sqcup\cW_{\bff\bfg\bfh}^{\mathrm{bal}},$$
where $\cW_{\bff\bfg\bfh}^f$ is the set of triples $(\nu_1,\nu_2,\nu_3)\in\cW_{\bff\bfg\bfh}^{\circ}$ of unbalanced weights with $\nu_1$ dominant, i.e.\ such that, if $\nu_i$ have weight $k_i$ for $i\in\{1,2,3\}$, then $k_1\geq k_2+k_3$. The sets $\cW_{\bff\bfg\bfh}^g$ and $\cW_{\bff\bfg\bfh}^h$ are defined similarly, with the weight $\nu_2$ and $\nu_3$ dominant respectively, and $\cW_{\bff\bfg\bfh}^{\mathrm{bal}}:=\{ (\nu_1,\nu_2,\nu_3)\in\cW_{\bff\bfg\bfh}^{\circ} \text{ of balanced weights} \}$.

Let $N:=\mathrm{lcm}(N_\bff,N_\bfg,N_\bfh)$  and define  $$S_{\Lambda_\bff}^{\ord}(N,\chi_f)[\bff]:=\{\ubff\in S_{\Lambda_\bff}^{\ord}(N,\chi_f) \mid T_\ell\ubff=a_\ell(\bff)\ubff \text{ for } \ell\nmid Np; \  U_p\ubff=a_p(\bff)\ubff \}$$ the set of \textit{$\Lambda$-adic test vectors} for $\bff$ (here $T_\ell$ and $U_p$ stand for the Hecke operators). Analogously we define $S_{\Lambda_\bfg}^{\ord}(N,\chi_g)[\bfg]$ and $S_{\Lambda_\bfh}^{\ord}(N,\chi_h)[\bfh]$.

For each choice of a triple of test vectors $(\ubff,\ubfg,\ubfh)$ for $(\bff,\bfg,\bfh)$, let  $$\L_p^g(\ubff,\ubfg,\ubfh)\in\Lambda_\bff\hat{\otimes}\mathrm{Frac}(\Lambda_\bfg)\hat{\otimes}\Lambda_\bfh$$ be the triple product $p$-adic $L$-function constructed in \cite{DR1}. It interpolates the square root of the central critical values $L(f_k\otimes g_\ell\otimes h_m,\frac{k+\ell+m-2}{2})$ as the triple of weights $(k,\ell,m)$ varies in $\cW_{\bff\bfg\bfh}^g$. 

 Hsieh \cite{Hsieh} constructed an explicit choice of test vector $(\ubff,\ubfg,\ubfh)$ for which  $\L_p^g(\ubff,\ubfg,\ubfh)$ actually belongs to $\Lambda_{\bff\bfg\bfh}:=\Lambda_\bff\hat{\otimes}\Lambda_\bfg\hat{\otimes}\Lambda_\bfh$ and it satisfies a simpler interpolation formula. We fix this choice of test vector once and for all. Moreover, while for the construction of \cite{DR1}, the specialisation at each classical point of $\bff,\bfg,\bfh$ has to be assumed to be old at $p$, in \cite{Hsieh} the specialisations of the three Hida families are allowed to be either old or new at $p$. We next summarize the interpolation properties of the triple product $p$-adic $L$-function attached to Hsieh's tests vectors. 

We recall that for an eigenform $\phi$ we denote by $\alpha_\phi$ and $\beta_\phi$ the two roots of the characteristic polynomial $x^2 - a_p(\phi)x + p^{k-1}\chi_\phi(p)$, ordered in such a way that $\ord_p(\alpha_\phi)\leq\ord_p(\beta_\phi)$. We will use the convention that, if $p$ divides the level of $\phi$, then $\beta_\phi=0$.

\begin{theorem}[Hsieh]\label{thm: Hsieh interpolation formula}
	Let $(\ubff,\ubfg,\ubfh)\in S_{\Lambda_\bff}^{\ord}(N,\chi_f)[\bff]\times S_{\Lambda_\bfg}^{\ord}(N,\chi_g)[\bfg]\times S_{\Lambda_\bfh}^{\ord}(N,\chi_h)[\bfh]$ be the triple of  $\Lambda$-adic test vectors for $(\bff,\bfg,\bfh)$ defined in \cite[Chapter 3]{Hsieh}. Then the $p$-adic $L$-function 
	$$\L_p^g(\ubff,\ubfg,\ubfh):\cW_{\bff\bfg\bfh}:=\mathrm{Spf}(\Lambda_{\bff\bfg\bfh})\longrightarrow\C_p$$ is uniquely characterised by the following interpolation property: for each $(k,\ell,m)\in\cW_{\bff\bfg\bfh}^g$  	
	\begin{equation*}
	\L_p^g(\ubff,\ubfg,\ubfh)^2(k,\ell,m)= \ L(f_k\otimes g_\ell\otimes h_m,c)\frac{\mathcal{E}(f_k,g_\ell,h_m)^2}{(-4)^\ell\langle g_\ell,g_\ell\rangle^2\mathcal{E}_0(g_\ell)^2\mathcal{E}_1(g_\ell)^2} \fa(k,\ell,m)\prod_{q\in\Sigma_\mathrm{exc}}(1+q^{-1}),
	\end{equation*}
	where:
	\begin{itemize}
\item $\langle \cdot ,\cdot \rangle$ is the Peterson product;
		\item $c=\frac{k+\ell+m-2}{2}$;
		\item $\Sigma_\mathrm{exc}$ is the set of exceptional primes defined in \cite[\S1.5]{Hsieh} 
		\item $\fa(k,\ell,m)=\Gamma_\C\Big(\frac{k+\ell+m-2}{2}\Big)\Gamma_\C\Big(\frac{-k+\ell-m+2}{2}\Big)\Gamma_\C\Big(\frac{k+\ell-m}{2}\Big)\Gamma_\C\Big(\frac{-k+\ell+m}{2}\Big)$, and $\Gamma_\C(s)=\frac{\Gamma(s)}{(2\pi)^s}$.
		\item $
		\mathcal{E}(f_k,g_\ell,h_m)=  \ (1-\beta_{g_\ell}\alpha_{f_k}\alpha_{h_m}p^{-c})(1-\beta_{g_\ell}\alpha_{f_k}\beta_{h_m}p^{-c})
		 \times(1-\beta_{g_\ell}\beta_{f_k}\alpha_{h_m}p^{-c})(1-\beta_{g_\ell}\beta_{f_k}\beta_{h_m}p^{-c});$
	
		\item $\mathcal{E}_0(g_\ell)= 1-\beta_{g_\ell}^2\chi_g^{-1}(p)p^{1-\ell};$
		\item $\mathcal{E}_1(g_\ell)=  1-\beta_{g_\ell}^2\chi_g^{-1}(p)p^{-\ell}.$
	\end{itemize}

\end{theorem}

\subsection{Kuga Sato varieties and the $p$-adic Abel--Jacobi map}\label{subsec: kuga sato}

\subsubsection{ Kuga--Sato varieties and modular forms}\label{sec: Modular forms and the cohomology of Kuga--Sato varieties}
Let $N>4$ be an integer and let $X_1(N)$ be the modular curve attached to the group $\Gamma_1(N)$. Let $\pi:\cE\rightarrow X_1(N)$ be the corresponding (generalized) universal elliptic curve and let $x$ be a non-cuspidal point of $X_1(N)$. Via the moduli interpretation of the modular curve, $x$ corresponds to the isomorphism class of a pair $(E_x,P_x)$ where $E_x$ is an elliptic curve of conductor $N$, and $P_x\in E_x$ is a point of exact order $N$. The fiber $\pi^{-1}(x)$ is isomorphic to the elliptic curve $E_x$. 

 The $r$-th \textit{Kuga--Sato variety} $W_r$ is the canonical desingularization of the $r$-th fibered product $\cE\times_{X_1(N)}\stackrel{(r)}{\cdots}\times_{X_1(N)}\cE$.
It is a variety of dimension $r+1$ defined over $\Q$. For a detailed description see \cite[Appendix]{BDP1}.

Let $f\in S_k(N,\chi)$ be a $p$-ordinary normalized cuspform of weight $k\geq2$ and field of Fourier coefficients $E_f$. Denote by $M_f$ the Chow  motive over $\Q$ and coefficients in $E_f$ attached to $f$ by Scholl \cite{Sch90}.  Recall that $M_f$ is given by the triple $(W_{k-2},e_f,0)$, where $e_f$ is a certain projector in the ring of correspondences of $W_{k-2}$, which is constructed from Hecke correspondences. By functoriality, $e_f$ acts on the different cohomological realizations of $M_f$ projecting onto the $f$-isotypical component. For example
\begin{align*}
  \coh^{k-1}_\etale((W_{k-2})_{\bar{\Q}},\Q_p)[f] =e_f\cdot\coh^{k-1}_\etale((W_{k-2})_{\bar{\Q}},\Q_p),
\end{align*}
which in fact is the $2$-dimensional $p$-adic Galois representation $V_f$ attached to $f$.

Similarly, we denote
\begin{align*}
  S_k(N)_L[f]=e_f\cdot S_k(N)_L,
\end{align*}
which is the projection onto the eigenspace of $f$ relative to the action of the Hecke operators $T_\ell$ with $(\ell,  N)=1$. 
For any number field $L$ containing $E_f$ the above $f$-isotypical component is isomorphic to a piece of the de Rham cohomology:
$$ S_k(N)_L[f] \simeq\fil^{k-1}\coh_\derham^{k-1}(W_{k-2}/L)[f],$$	
and we denote by $\omega_f$ the element of $\fil^{k-1}\coh_\derham^{k-1}(W_{k-2}/\C_p)$ corresponding to $f$ via the previous isomorphism and our chosen inclusion $L\subseteq \C_p$.

Assume that $\ord_p(N)\leq1$, so that the Kuga--Sato variety $W_{k-2}$ has good or semistable reduction at $p$ and let $\sigma_p$ denote a geometric Frobenius element at $p$.
The cohomology space $\coh_\derham^{k-1}(W_{k-2}/\C_p)[f]$ associated to $M_f$  has dimension $2$ over $\C_p$ and $\sigma_p$ acts on it.
Since $f$ is ordinary, there is a \textit{unit-root} subspace $$\coh_\derham^{k-1}(W_{k-2}/\C_p)[f]^\unitroot\subseteq \coh_\derham^{k-1}(W_{k-2}/\C_p)[f]$$ of dimension $1$ on which $\sigma_p$ acts as a $p$-adic unit. Define $\eta_f$ to be the unique element in the space $\coh_\derham^{k-1}(W_{k-2}/\C_p)[f]^\unitroot$  such that $\langle \eta_f,\omega_f \rangle=1$, where  $\langle\cdot,\cdot\rangle$ denotes the Poincare pairing on  $\coh_\derham^{k-1}(W_{k-2}/\C_p)$.

\subsubsection{The $p$-adic Abel--Jacobi map}
Let $W=W_r$ be the $r$-th Kuga--Sato variety of level $N$. We denote by $\CH^c(W)$ the Chow group of rational equivalence classes of codimension $c$ cycles on $W$, and by $\CH^c(W)_0$ the subgroup of classes of null-homologous cycles, i.e.\ the kernel of the cycle class map. If $K$ is an extension of $\Q$, we denote by $\CH^c(W/K)_0$ the group of null-homologous cycles defined over $K$. Also, if $L$ is a number field  we will denote by $\CH^c(W/K)_{0,L}$ the space $L\otimes_\Z \CH^c(W/K)_0$.

Fix a prime $\fp$ of $K$ above $p$ and denote by $W_{K_\fp}$ the base extension of $W$ to the completion $K_\fp$ of $K$ at $\fp$. As in \S\ref{sec: Modular forms and the cohomology of Kuga--Sato varieties}, we assume that $\ord_p(N)\leq 1$, so that $W_{K_\fp}$ has either good or semistable reduction. In both situations (cf. \cite[\S 3.4]{BDP1} for the good reduction case and \cite[\S 2]{Cas18} for the semistable case) for any $c\in \{0,\dots,r+1\}$ there exists a so called \emph{ $p$-adic Abel--Jacobi map}
 $$\AJ_p: \CH^c(W/K_\fp)_{0,\Q}\longrightarrow\fil^{c}\coh^{2c-1}_\derham(W_{K_\fp}/K_\fp)^\lor.$$
Here ${}^\vee$ denotes $K_\fp$-dual.

\begin{remark}\label{rmk:log}
The $p$-adic Abel--Jacobi map can be seen as a generalization of the formal group logarithm attached to a differential form on $X_1(N)$. Indeed, one has that $\log_{\omega_f}(P)=\AJ_p(P)(\omega_f)$ for any $P\in X_1(N)(\Q_p)$.   
\end{remark}

\subsection{The conjecture}\label{sec: the conjecture}
Let $f\in S_k(N_f,\chi_f), \ g\in M_1(N_g,\chi_g), \ h\in M_1(N_h,\chi_h)$ be three normalised eigenforms, with $k\geq2$ and $\chi_f\cdot\chi_g\cdot\chi_h=1$. Fix a prime number $p$ such that $p\nmid N_gN_h$ and $\ord_p(N_f)\leq 1$. Assume that $f$, $g$, and $h$ are $p$-ordinary, and set $N$ to be the prime-to-$p$-part of $\mathrm{lcm}(N_f,N_g,N_h)$. We begin this section by defining, under certain additional conditions, a regulator $\mathrm{Reg}(f,g_\alpha,h)$ which generalizes to weight $k\geq 2$ the one defined in \cite{DLR} for $k=2$, where we recall that $g_\alpha$ stands for the $p$-stabilization of $g$ such that $U_p g_\alpha = \alpha_gg$.

 Let $\rho_g$ (resp. $\rho_h$) denote the Artin representation attached to $g$ (resp. $h$), regarded as acting on a $2$-dimensional $L$-vector space $V_g$ (resp. $V_h$). Let also $\rho_{gh}$ denote the tensor product representation acting on $V_{gh}=V_g\otimes V_h$, and let $H$ be the field fixed by its kernel, so that we have:
 \begin{align*}
   \rho_{gh}\colon \Gal(H/\Q)\lra \Aut_L(V_{gh}).
 \end{align*}
We can assume, extending $L$ if necessary, that $L$ contains the Fourier coefficients of $f,g$ and $h$.
We denote by $H_p$ the completion of $H$ in $\Qbar_p$ induced from our fixed inclusion $\Qbar\subset \Qbar_p$. Since $p\nmid N_gN_h$, $H_p$ is unramified and we denote by $\sigma_p$ a geometric Frobenius.

The Elliptic Stark Conjecture of \cite{DLR} is formulated under a certain classicality hypothesis for $g$, labeled as Hypothesis C in loc.\ cit., which we will also assume. In fact, we will assume the following more explicit condition, labeled as Hypothesis C' in \cite{DLR}, which implies (and is presumably equivalent to) Hypothesis C.
\begin{assumption}\label{ass: hypothesis c'} The modular form $g$ satisfies one of the following conditions:
	\begin{enumerate}[\indent (1)]
		\item It is a cuspform \textit{regular} at $p$ (i.e.\ $\alpha_g\neq\beta_g$), and it is not the theta series of a character of a real quadratic field in which $p$ splits;
		\item  it is an Eisenstein form which is \textit{irregular} (i.e.\ $\alpha_g=\beta_g$).
	\end{enumerate}
\end{assumption}
This assumption allows for the definition of a $1$-dimensional $L$-subspace $V_g^\alpha$ of $V_g$ as in \cite[\S1]{DLR}:
	\begin{itemize}
		\item If $g$ satisfies the first condition, then the attached Artin representation $V_g$ decomposes as the direct sum  of the eigenspaces $V_g^\alpha,V_g^\beta$ with respect to the action of $\sigma_p$, with eigenvalues $\alpha_g\chi_g(p)^{-1}={\beta_g}^{-1},\beta_g\chi_g(p)^{-1}=\alpha_g^{-1}$ respectively. 
		\item If $g$ is an irregular Eisenstein form, we take $V_g^\alpha$ to be any $1$-dimensional subspace of $V_g$ which is not stable under the action of $G_{\Q}$.
	\end{itemize}	
From now on, we will also assume the following assumption on the local signs of $L(f\otimes g\otimes h,s)$.
\begin{assumption}\label{ass: local signs + vanishing at central point}
	The local signs $\epsilon_v(f,g,h)$ of (\ref{eq: sign of the functional equation}) at finite primes $v\mid \mathrm{lcm}(N_f,N_g,N_h)$ are $+1$.
\end{assumption}
Note that, since the weights $(k,1,1)$ of the triple $(f,g,h)$ are unbalanced, then $\epsilon_\infty(f,g,h)=+1$. Hence the global sign  is $+1$ and the order of vanishing of $L(f\otimes g\otimes h,s)$ at the central point $k/2$ is even. In order to simplify a bit the notation for the $f$-isotypical component, we put
\begin{align*}
  \CH^{k/2}(W_{k-2}/H)_{0,L}^{[f]}:= e_f\cdot \CH^{k/2}(W_{k-2}/H)_{0,L}.
\end{align*}
 As we mentioned in \S \ref{sec: intro}, a conjecture due to Beilinson predicts that
\begin{align*}
\ord_{s=k/2}L(f\otimes g\otimes h,s) \stackrel{?}{=} \dim_L\Hom_{G_\Q}(V_{gh},\CH^{k/2}(W_{k-2}/H)_{0,L}^{[f]}).
\end{align*}
If $\dim_L\Hom_{G_\Q}(V_{gh},\CH^{k/2}(W_{k-2}/H)_{0,L}^{[f]})= 2$, then we can define a regulator by fixing  an $L$ basis  $(\Phi_1,\Phi_2)$ of this space and an $L$-basis $(v_1,v_2)$ of $V_{gh}^\alpha:=V_g^\alpha\otimes V_h$ as follows. 
\begin{definition}
	The \textit{regulator} attached to the triple $(f,g_\alpha,h)$ is 
	\begin{align*}\mathrm{Reg}(f,g_\alpha,h):=\det\mtx{\AJ_p(\Phi_1(v_1))(\omega_f)}{\AJ_p(\Phi_1(v_2))(\omega_f)}{\AJ_p(\Phi_2(v_1))(\omega_f)}{\AJ_p(\Phi_2(v_2))(\omega_f)},\end{align*}
where $\omega_f\in\coh_\derham(W_{k-2}/\C_p)[f]$ is the class defined in \S \ref{sec: Modular forms and the cohomology of Kuga--Sato varieties}.
\end{definition}
Observe that, by Remark \ref{rmk:log}, when $k=2$ we recover the regulator as defined in \cite{DLR}.

Let $\mathrm{Ad}_g$ be the adjoint representation of $\rho_g$ and let $H_g$ be the field fixed by its kernel. Let $u_{g_\alpha}\in \left( L\otimes_\Z (\cO_{H_g}[1/p]^\times)\right)^{\mathrm{Ad_g}}$ be the Stark unit defined in \cite[\S 1.2]{DLR}, on which $\sigma_p$ acts with eigenvalue $\alpha_g/\beta_g$. The following is a generalization to weights $k\geq 2$ of \cite[Conjecture ES]{DLR}. In the statement we use the following notation: If $\phi$ is a modular form of weight $w$ and level $Np$ which is an eigenform for the good Hecke operators then $M_w(Np)[\phi]$ denotes the isotypical subspace of $M_k(N)$ defined as
\begin{align*}
  M_w(Np)[\phi] = \{\breve\phi\in M_w(Np)\colon T_\ell\breve\phi = a_\ell(\phi) \breve \phi\text{ for all }\ell\nmid Np\}.
\end{align*}
If $\phi$ happens to be also an eigenform for $U_p$, then $M_w(Np)[\phi]$ will be understood as the isotypical subspace associated to the good Hecke operators and to $U_p$ as well.  
\begin{conjecture}\label{conj: general weights} Let $\bff,\bfg,\bfh$ be the Hida families passing through the $p$-stabilisations $f_\alpha, g_\alpha, h_\alpha$ of $f,g$ and $h$. Set $r = \dim_L\Hom_{G_\Q}(V_{gh},\CH^{k/2}(W_{k-2}/H)_{0,L}^{[f]})$.
	\begin{enumerate}[\indent $i)$]
		\item If $r> 2$, then $\L^g_p(\ubff,\ubfg,\ubfh)(k,1,1)=0$ for any choice of test vectors $(\ubff,\ubfg,\ubfh)$ for $(\bff,\bfg,\bfh)$;
		\item if $\ord_{s=k/2}L(f\otimes g\otimes h,s)=2$, then there exist a triple of test vectors
\begin{align*}
(\breve f, \breve g_\alpha, \breve h)\in S_k(Np, \chi_f)_L[f]\times M_1(Np,\chi_g)_L [g_\alpha]\times M_1(Np,\chi_h)_L[h]  
\end{align*}
and Hida families $\ubff,\ubfg,\ubfh$ with $f_k = \breve f$, $g_1 = \breve g_\alpha$, $h_1 = \breve h$,  such that
		\begin{equation}\label{eq: conjecture (k,1,1)}
		\L^g_p(\ubff,\ubfg,\ubfh)(k,1,1)=\dfrac{\mathrm{Reg}(f,g_\alpha,h)}{\mathfrak{g}(\chi_f)\log_p(u_{g_\alpha})},
		\end{equation}  
		where $\mathfrak{g}(\chi_f)$ denotes the Gauss sum of the character $\chi_f$.
	\end{enumerate}
\end{conjecture}

\begin{remark}
	Note that, since $\L^g_p(\ubff,\ubfg,\ubfh)(k,1,1)$ belongs to $L_p\subset H_p\otimes L$, the geometric Frobenius $\sigma_p$ acts trivially on the left-hand part of \eqref{eq: conjecture (k,1,1)}. On the other hand, $\sigma_p(\log_p(u_{g_\alpha}))=\alpha_g/\beta_g$ and by definition of the space $V_{gh}^\alpha$, the element $\sigma_p$ acts on the regulator as multiplication by $(\beta_g\beta_h)^{-1}\cdot(\beta_g\alpha_h)^{-1}$; finally,  $\sigma_p$ acts on the Gauss sum with eigenvalue $\chi_f(p)$. So $\sigma_p$ also acts trivially on the right hand side of \eqref{eq: conjecture (k,1,1)} as $\beta_g\alpha_g\beta_h\alpha_h\chi_f(p)=\chi_g(p)\chi_h(p)\chi_f(p)=1.$
\end{remark}

 \section{A particular case of the Elliptic Stark Conjecture in weights $(k,1,1)$}\label{sec: factorization formula}
The goal of this section is to provide theoretical evidence in support of Conjecture \ref{conj: general weights}  in the particular case where $g$ and $h$ are theta series of the same imaginary quadratic field in which $p$ splits. The main result that we prove is Theorem \ref{thm: special case weight (k,1,1)}, which relates the triple product $p$-adic $L$-function in this setting with the $p$-adic Abel--Jacobi image of certain Heegner cycles. We begin by reviewing Heegner cycles in \S\ref{sec: Generalised Heegner cycles}  (and, in fact, we will describe the so called generalized Heegner cycles introduced in the works of Bertolini--Darmon--Prasanna, since these more general cycles will appear in \S \ref{sec: general unbalanced weights}). In \S \ref{sec: statement of the particular case} we particularize the Elliptic Stark Conjecture to the case of theta series of imaginary quadratic fields, and we state the main result. The proof is given in \S \ref{sec: a factorization formula}, and it follows from a factorization formula for the triple product $p$-adic $L$-function in this case. The definition and the main properties of the $p$-adic $L$-functions involved are recalled in \S\ref{sec: background on p-adic L-functions}.

We fix from now on an imaginary quadratic field $K$ of discriminant $-D_K$. We denote by $h_K$ its class number and by $\cO_K$ its ring of integers. 

 \subsection{Generalised Heegner cycles}\label{sec: Generalised Heegner cycles}  
 Let $N$ be a squarefree positive integer coprime to $D_K$. From now on we will assume the following \textit{Heegner Hypothesis} for the pair $(K,N)$.
 \begin{assumption}\label{ass: HH}
There exists an ideal $\cN$ of $\cO_K$ coprime to $D_K$ such that $\cO_K/\cN\cong\Z/N\Z$.
 \end{assumption}
 
 Fix an elliptic curve $A$ over the Hilbert class field of $K$ and with complex multiplication by $\cO_K$, and a generator $t$ of $A[\cN]$ so that the pair $(A,t)$ corresponds to a point $P$ on the modular curve $X_1(N)$. 
 In \cite{BDP1} and \cite{BDP2}, Bertolini, Darmon and Prasanna constructed a family of so called \textit{generalised Heegner cycles} in the product of a Kuga--Sato variety with  a power of $A$. As we will recall in \S\ref{sec: p-adic rankin L-function}, these cycles are related to special values of a $p$-adic $L$-function, and we will use this relation in \S\ref{sec: proof of the special case} to prove a special case of Conjecture \ref{conj: general weights}. We now briefly recall the definition of the cycles.
 
 Let $c$ be positive integer coprime to $N D_K$, and let $\cO_c:=\Z + c\cdot\cO_K$  be the order of $K$ of conductor $c$.  Let $A_c:=\C/\cO_c$ be an elliptic curve with complex multiplication by $\cO_c$, which we can assume is defined over the ring class field $K_c$ of $K$ of conductor $c$. Let $\phi_c:A\longrightarrow A_c$ be an isogeny of degree $c$.
 Given an ideal $\fa$ of $\cO_c$ prime to $\cN_c:=\cN\cap\cO_c$, denote by $A_\fa$ the elliptic curve $\C/\fa^{-1}$ and by $\phi_\fa$ the isogeny  $$\phi_\fa:A_c\longrightarrow A_\fa.$$ 
 The isogeny $\phi_\fa\circ\phi_c$ defines a $\Gamma_1(N)$-level on $A_\fa$, i.e.\ a point $t_\fa:=\phi_\fa\circ\phi_c(t)$ of exact order $\cN_c$.  
 
 Let $r_0\geq r_1$ be two non-negative integers with the same parity, set $s:=\frac{r_0+r_1}{2}$, $u:=\frac{r_0-r_1}{2}$ and let 
$$X_{r_0,r_1}:=W_{r_0}\times A^{r_1}.$$ It is a variety of dimension $r_0+r_1+1=2s+1$ defined over the Hilbert class field $K_1$ of $K$. 

 Let $$\pi:X_{r_0,r_1}\overset{p_1}{\longrightarrow}W_{r_0}\longrightarrow X_1(N)$$ be the composition of the projection on the first component of $X_{r_0,r_1}$ with the canonical map of the Kuga--Sato variety onto $X_1(N)$. For each ideal $\fa$ of $\cO_c$ prime to $\cN_c$, the fiber of the point $P_\fa=(A_\fa,t_\fa)$ is $$\pi^{-1}(P_\fa)\cong A_\fa^{r_0}\times A^{r_1}\cong(A_\fa\times A)^{r_1}\times(A_\fa\times A_\fa)^u.$$
 
Write $\mathrm{End}(A_\fa)$ as $\Z[\frac{d_c+\sqrt{d_c}}{2}]$, where we regard $\sqrt{d_c}$ as an endomorphism of the curve, and define $\Gamma_\fa:=(\mathrm{Graph}(\sqrt{d_c}))^{\mathrm{tr}}\subset A_\fa\times A_\fa$, where tr denotes the transpose. Let also $\Gamma_{c,\fa}:=\mathrm{Graph}(\phi_\fa\circ\phi_c)^{\mathrm{tr}}$, which is a cycle in $A_\fa\times A$, and $\Gamma_{r_0,r_1,c,\fa}:=\Gamma_{c,\fa}^{r_1}\times\Gamma_\fa^u$, which is a cycle of codimension $s+1$ in $X_{r_0,r_1}$ supported on the fiber $\pi^{-1}(P_\fa)$. The \emph{generalised Heegner cycle} attached to the data $r_0,r_1,\mathfrak a,c$ is defined as $$\Delta_{r_0,r_1,c,\fa}:=\epsilon_{X_{r_0,r_1}}(\Gamma_{r_0,r_1,c,\fa})\in\CH^{s+1}(X_{r_0,r_1})_{0,\Q},$$  where $\epsilon_{X_{r_0,r_1}}$ is the projector defined in \cite[\S 4.1]{BDP2}.

Let now $f\in S_{r_0+2}(N,\chi_f)$ be a modular form. We want to consider the projection to the ``$(f,\psi)$-component'' of these cycles, for certain Hecke characters $\psi$. Recall that a \textit{Hecke character} of $K$ of infinity type $(\ell_1,\ell_2)\in \Z^2$ is a continuous homomorphism $\psi:\A^\times_K\rightarrow\C^\times$
 such that $$\psi(\alpha\cdot x\cdot x_\infty)=\psi(x)\cdot x_\infty^{-\ell_1}\cdot\bar{x}_\infty^{-\ell_2},\ \text{ for all } \alpha\in K^\times,x \in \A_K^\times,\text{ and } x_\infty \in \C.$$ 
The \emph{conductor} of $\psi$ is the largest ideal $\fc_\psi$ of $K$ such that, for each prime ideal $\fq$ of $K$, the $\fq$-component of $\psi$ is trivial when restricted to $1+\mathfrak{c}_\psi\cO_{K_\fq}$.
The \textit{central character} of $\psi$ is the Dirichlet character  $\varepsilon_\psi:=\psi_{|\A_\Q^\times}\cdot \norm^{-\ell_1-\ell_2}$, where $\norm$ stands for the norm character.

Let $\varepsilon$ be a Dirichlet character of conductor $N_\varepsilon\mid N_f$ and let $\cN_\varepsilon$ be the ideal of $\cO_K$ such that $\cN_\varepsilon\mid\cN_f$ and $\cO_K/\cN_\varepsilon\cong\Z/N_\varepsilon\Z$.
Let $c$ be a positive integer such that $\gcd(c,N_f\cdot D_K)=1$, and let $U_c:=\hat{\cO}_c^\times=\prod_p(\cO_c\otimes\Z_p)^\times$.
A Hecke character $\psi$ of $K$ is \textit{of finite type $(c,\cN_f,\varepsilon)$} if the conductor of $\psi$ is $c\cN_\varepsilon$ and if, denoting $\cN_{c,\varepsilon}:=\cN_\varepsilon\cap\cO_c$, the restriction of $\psi$ to $U_c$ coincides with the  composition
$$U_c=\hat{\cO}_c^\times\longrightarrow\Big(\hat{\cO}_c/\cN_{c,\varepsilon}\hat{\cO}_c \Big)^\times\cong\Big(\cO_K/\cN_\varepsilon\cO_K \Big)^\times\cong\Big(\Z/N_\varepsilon\Z \Big)^\times\xrightarrow{\varepsilon^{-1}}\C^\times.$$  
Let $\psi$ be a Hecke character of $K$ of infinity type $(r_1-j,j)$ for some $0\leq j\leq r_1$ and such that $\varepsilon_\psi=\chi_f$. The condition on the central character implies that the complex $L$-function $L(f,\psi^{-1},s)$ is selfdual. Let $\cN_{\chi_f}$ be the ideal of $K$ dividing $\cN$ whose norm equals the conductor of $\chi_f$, and suppose that the conductor of $\psi$ is of the form $\mathfrak{c}_\psi=c \cN_{\chi_f}$, for some $c\in\Z$ coprime to $N$. 

By \cite[\S 4.2]{BDP2} the cycle $\Delta_{r_0,r_1,c,\fa}$ is then defined over the number field $F$ that corresponds by class field theory to the subgroup $K^\times W\subseteq \A_K^\times$, where 
\begin{align*}
  W = \{x\in \A_K^\times \colon x\cO_c =\cO_c ,\, xt = t\}.
\end{align*}
Arguing as in \cite[Display (4.2.1)]{BDP2} we find that $\Gal(F/K_c)\simeq (\Z/N\Z)^\times/\{\pm 1\}$. Define $H_{c,f}$ to be the subextension of $F/K$ that corresponds to $\ker \chi_f$ under this isomorphism and we let
\begin{align}\label{eq: gen heegner cycle}
  \tilde \Delta_{r_0,r_1,\mathfrak a, c} :=  \frac{1}{[F\colon H_{c,f}]}\operatorname{Tr}_{F/H_{c,f}}( \Delta_{r_0,r_1,\mathfrak a, c}).
\end{align}
 Finally, following \cite[Definition 4.2.3]{BDP2} define the following cycle:
  \begin{align}\label{eq: nuestro ciclo de heegner generalizado}
    \tilde \Delta_{r_0,r_1,c}^\psi:= e_f\left(\sum_{\mathfrak a \in S} \psi(\mathfrak a)^{-1}\cdot  \tilde \Delta_{r_0,r_1,\mathfrak a, c}\right),
  \end{align}
where $S$ is a set of representatives for $\Pic(\cO_c)$ that are prime to $c\cdot \cN$. This definition might depend on the choice of $S$, but its image under the $p$-adic Abel--Jacobi does not by \cite[Remark 4.2.4]{BDP2}. Observe that $\tilde \Delta_{r_0,r_1,c}^\psi$ belongs to $\CH^{s+1}(X_{r_0,r_1}/H_{c,f})_{0,\Q(\psi)}$, where $\Q(\psi)$ denotes the number field generated by the values of $\psi$.

\subsection{The case of theta series of imaginary quadratic fields}\label{sec: statement of the particular case} Let $f\in S_k(N_f,\chi_f)$ be a normalised newform of weight $k\geq2$. Suppose that the level $N_f$ is squarefree, coprime to $D_K$ and that satisfies the Heegner Hypothesis for $K$ (cf. \ref{ass: HH}). Fix from now on an ideal $\cN_f$ of $\cO_K$ of norm $N_f$ and let $\cN_{\chi_f}\mid\cN_f$ be the ideal whose norm is the conductor of $\chi_f$. Let also $\psi_g,\psi_h\colon \A_K^\times\ra \C^\times$ be finite order Hecke characters of conductors $\fc_g,\fc_h$ and central characters $\varepsilon_{g}, \varepsilon_h$, and let $g$ and $h$ be the theta series attached to $\psi_g$ and $\psi_h$, respectively. They are weight $1$ modular forms with Nebetypus; more precisely, if $\chi_K$ denotes the quadratic character attached to $K$, then
$$g\in M_1(N_g,\chi_g) \ \ \text{and} \ \ h\in M_1(N_h,\chi_h),$$ where
$$N_g=D_K\cdot\norm_{K/\Q}(\fc_g), \  \ N_h=D_K\cdot\norm_{K/\Q}(\fc_h), \ \ \chi_g=\chi_K\cdot\varepsilon_{g}, \ \ \chi_h=\chi_K\cdot\varepsilon_{h}.$$
Fix a prime number $p\geq 3$ that splits in $K$ as $p\cO_K=\p\bar{\p}$. Suppose that $p\nmid N_g\cdot N_h$ and $\ord_p(N_f)\leq1$ and suppose that $f, g$ and $h$ are $p$-ordinary. We assume that 
\begin{align}
  \label{eq:characters}
\chi_f\cdot\chi_g\cdot\chi_h=1, 
\end{align}
so that we are in the setting of \S\ref{sec: the conjecture}. In this section we are interested in Conjecture \ref{conj: general weights} for the modular forms $f,g$, and $h$ we just defined. Recall that the field of coefficients $L$ can be taken to be the field generated by the Fourier coefficients of $f$, $g$, and $h$.

Let $\psi_h'$ be the Hecke character defined by  $\psi_h'(\sigma)=\psi_h(\sigma_0\sigma\sigma_0^{-1})$ for $\sigma\in G_\Q$, where $\sigma_0$ is any lift of the non-trivial involution of $K/\Q$. Define also the characters
\begin{align}
  \label{eq:psi1psi2}
  \psi_1 := \psi_g\psi_h,\ \ \psi_2 := \psi_g\psi_h'.
\end{align}

Condition \eqref{eq:characters} implies that ${\psi_1}_{|\A_\Q^\times}={\psi_2}_{|\A_\Q^\times}=\chi_f^{-1}$ and that the conductor of $\psi_i$ is of the form $c_i\cN_i$, where $\cN_i\mid\cN_{\chi_f}$ and $c_i$ is an integer coprime to $\cN_{\chi_f}$.
We will assume from now on that $\cN_i=\cN_{\chi_f}$ for $i=1,2$. We further assume that $\psi_i$ has finite type $(c_i,\cN_f,\chi_f)$.

 In this setting, by looking at the Euler factors one checks that there is a decomposition of Artin representations
\begin{align}\label{eq:dec of artin reps}
  V_{gh}=V_{\psi_1}\oplus V_{\psi_2},
\end{align}
which in turn induces a factorization of $L$-functions
\begin{equation}\label{eq: factorisation of complex L functions weights (k,1,1)}
L(f\otimes g\otimes h,s)=L(f/K,\psi_1,s)L(f/K,\psi_2,s).
\end{equation}
The conditions imposed so far imply that all the finite local signs of $L(f/K,\psi_i,s)$ are $+1$, so in particular Assumption \ref{ass: local signs + vanishing at central point} is satisfied. Moreover, the global sign of $L(f/K,\psi_i,s)$ is $-1$, so that $L(f/K,\psi_i,s)$  vanishes at the central point $s=k/2$ and therefore $\ord_{s=k/2}L(f\otimes g\otimes h,s) \geq 2.$

Thanks to our assumptions on the characters $\psi_1$ and $\psi_2$, we can speak of the cycles $\tilde \Delta_{k-2,0,c_i}^{\bar\psi_i}$ as defined in Section \ref{sec: Generalised Heegner cycles}. Observe that in this particular case in which $r_1=0$, these are in fact classical Heegner cycles, as defined in \cite{Nekovar}. If we let $H_i$ be the field denoted as $H_{c_i,\chi_f}$ in \S\ref{sec: Generalised Heegner cycles}, then we have that
 \begin{align*}
 \tilde \Delta_{k-2,0,c_i}^{\bar\psi_i}\in \CH^{k/2}(W_{k-2}/H_i)_{0,L}^{[f]}.
 \end{align*}
Since we want to view the two cycles as being defined over the same field, we set $H:=H_1\cdot H_2$, the composition of $H_1$ and $H_2$, and $c:=\mathrm{lcm}(c_1,c_2)$. In this way, the cycle $\tilde \Delta_{k-2,0,c_i}^{\bar\psi_i}$ belongs to $\CH^{k/2}(W_{k-2}/H)_{0,L}^{[f]}$.

The following is the main result of this section. It is the generalization to weights  $k\geq 2$ of \cite[Theorem 3.3]{DLR}.

\begin{theorem}\label{thm: special case weight (k,1,1)}
Let $\bff$, $\bfg$, $\bfh$ be Hida families passing through $f_\alpha$, $g_\alpha$ and $h_\alpha$. Let $(\ubff,\ubfg,\ubfh)$ be the triple of test vectors of Theorem \ref{thm: Hsieh interpolation formula}. There exists a quadratic extension $L_0$ of $L$ and $\lambda\in L_0^\times$ such that
	$$\L_p^g(\ubff,\ubfg,\ubfh)(k,1,1)=\lambda\, \frac{\mathrm{AJ}_p(\tilde \Delta_{k-2,0,c_1}^{\bar\psi_1})(\omega_f)\mathrm{AJ}_p(\tilde \Delta_{k-2,0,c_2}^{\bar\psi_2})(\omega_f)}{\mathfrak{g}(\chi_f)\log_p(u_{g_\alpha})} .$$
\end{theorem}
\begin{remark}
We stress that the non-zero scalar $\lambda$ lies in a quadratic extension of the field of coefficients of $f$, $g$, and $h$. In this sense, this also represents a slight strengthening of \cite[Theorem 3.3]{DLR}, in which one had a less precise control of the degree of such extension.   
\end{remark}

Theorem \ref{thm: special case weight (k,1,1)} can be seen as giving evidence towards Conjecture \ref{conj: general weights}, as we now explain. The decomposition of representations \eqref{eq:dec of artin reps} induces a decomposition

\begin{align*}
&\hom_{\mathrm{G}_\Q}(V_{gh},\mathrm{CH}^{k/2}(W_{k-2}/H)_{0,L}^{[f]})\simeq\\ &\hom_{\mathrm{G}_\Q}(V_{\psi_1},\mathrm{CH}^{k/2}(W_{k-2}/H)_{0,L}^{[f]})\oplus\hom_{\mathrm{G}_\Q}(V_{\psi_2},\mathrm{CH}^{k/2}(W_{k-2}/H)_{0,L}^{[f]}),
\end{align*}
where $V_{\psi_i}$ denotes the $2$-dimensional representation over $L$ obtained by induction from $\psi_i$. 

Consider now the Heegner cycle $\tilde\Delta_c:=\tilde \Delta_{k-2,0,\cO_c,c}$ defined in \eqref{eq: gen heegner cycle} with $r_0=k-2$, $r_1=0$ and associated to the trivial ideal $\cO_c$. It belongs to $\mathrm{CH}^{k/2}(W_{k-2}/H)_{0,L}^{[f]}$, and we consider the projection to the $\bar\psi_i$ component
\begin{align}\label{eq: projection}
\Delta_c^{\bar\psi_i}:=  \sum_{\sigma\in\Gal(H/K)}\psi_i(\sigma)(\tilde \Delta_{c})^\sigma.
\end{align}

Observe that $\Delta_c^{\bar\psi_i}$ gives an element in $\hom_{\mathrm{G}_\Q}(V_{\psi_i},\mathrm{CH}^{k/2}(W_{k-2}/H)_{0,L}^{[f]})$. Indeed, since $\psi_i$ is anticyclotomic we have that $V_{\psi_i}=V_{\bar\psi_i}$, and  by Frobenius reciprocity giving an element in $$\hom_{\mathrm{G}_\Q}(V_{\bar\psi_i},\mathrm{CH}^{k/2}(W_{k-2}/H)_{0,L}^{[f]})$$ is equivalent to giving a $G_K$-homomorphism from $L$ (viewed as a $G_K$-module via the action of $\bar\psi_i$) to $\mathrm{CH}^{k/2}(W_{k-2}/H)_{0,L}^{[f]}$; the map that sends $1$ to $\Delta_c^{\bar\psi_i}$ is one such homomorphism.

As we mentioned before, in our setting the sign of the functional equation of $L(f,\psi_i,s)$ is $-1$ and therefore these $L$-functions vanish at the central point $s=k/2$. Suppose that $\ord_{s=k/2}L(f,\psi_i,s)=1$ for $i=1,2$ (a fact which is actually expected to hold for ``generic'' Hecke characters). Then, by the general philosophy of Heegner points and Heegner cycles it is expected that
\begin{align}\label{eq: assumption on heegner cycles}
  \hom_{\mathrm{G}_\Q}(V_{\psi_i},\mathrm{CH}^{k/2}(W_{k-2}/H)_{0,L}^{[f]})=\langle \Delta_c^{\bar\psi_i} \rangle.
\end{align}
That is to say, the above space is generated by the homomorphism given by the Heegner cycle. This has been proven for $k=2$ by the results of Gross--Zagier \cite{GZ}, Kolyvagin \cite{kolyvagin}, Zhang \cite{zhang}, and Bertolini--Darmon \cite{BDKo}. For $k>2$, in the particular case where $c=1$ and $\chi_f=1$, it follows from results of Zhang \cite{zhang-cycles} and Nekovar \cite{Nekovar} on Heegner cycles if one assumes the Gillet-Soulé conjecture on the non-degeneracy of the height pairing.

Arguing as in \cite[Lemma 3.2]{DLR} we see that if we assume \eqref{eq: assumption on heegner cycles} then 
\begin{align*}
  \mathrm{Reg}(f,g_\alpha,h)=\AJ_p({\Delta}_{c}^{\bar\psi_1})(\omega_f)\cdot \AJ_p({\Delta}_{c}^{\bar\psi_2})(\omega_f).
\end{align*}
Finally, observe that $\Delta_c^{\bar\psi_i}$ and $\tilde \Delta^{\bar\psi_i}_{k-2,0,c_i}$ are defined differently (see \eqref{eq: projection}  and \eqref{eq: nuestro ciclo de heegner generalizado}), but we have that 
\begin{align}\label{eq: equality of AJs}
  \AJ_p({\Delta}_{c}^{\bar\psi_i})(\omega_f) = [H:K_1]\cdot \AJ_p({\Delta}_{k-2,0,c_i}^{\bar\psi_i})(\omega_f)\pmod{L^\times}.
\end{align}
Indeed, by Shimura's reciprocity law $(\tilde \Delta_{r_0,r_1,\cO_c, c})^\sigma = \tilde \Delta_{r_0,r_1,\mathfrak a^{-1}, c} $ when $\sigma$ corresponds to $\fa$ under the reciprocity map of class field theory. Then \eqref{eq: equality of AJs} follows from the display in the proof of \cite[Proposition 4.2.1]{BDP2} and the fact that $\psi_i$ is of finite type $(c_i,\cN_f,\chi_f)$. 

Therefore we see that
\begin{align*}
  \mathrm{Reg}(f,g_\alpha,h)=\AJ_p({\Delta}_{c}^{\bar\psi_1})(\omega_f)\cdot \AJ_p({\Delta}_{c}^{\bar\psi_2})(\omega_f) \pmod{L^\times}.
\end{align*}
and therefore Theorem \ref{thm: special case weight (k,1,1)} proves Conjecture \ref{conj: general weights} in this case (up to the fact that in Theorem \ref{thm: special case weight (k,1,1)} $\lambda$ lies in a quadratic extension of $L$ rather than in $L$ itself).

We devote the rest of Section \ref{sec: factorization formula} to prove Theorem \ref{thm: special case weight (k,1,1)}. The argument follows essentially the same strategy introduced in \cite[\S 3]{DLR}, which exploits a certain factorization of $p$-adic $L$-functions. In \S \ref{sec: background on p-adic L-functions} we will recall the different $p$-adic $L$-functions involved. Then,  in \S \ref{sec: a factorization formula} we will state and prove the factorization formula, and we will prove Theorem \ref{thm: special case weight (k,1,1)}.

\subsection{$p$-adic $L$-functions}\label{sec: background on p-adic L-functions}
In this subsection we review two types of $p$-adic $L$-functions: the Bertolini--Darmon--Prasanna $p$-adic $L$-function (and Castell\`{a}'s generalization) and Katz's $p$-adic $L$-function.
\subsubsection{The Bertolini--Darmon--Prasanna $p$-adic $L$-function}\label{sec: p-adic rankin L-function}
Let $f\in S_k(N_f,\chi_f)$ be a normalized newform and let $\psi$ be a Hecke character of the imaginary quadratic field $K$ of infinity type $(\ell_1,\ell_2)$, conductor $\fc_\psi$ and central character $\varepsilon_\psi$. We can attach to the pair $(f,\psi)$ the complex $L$-function $$L(f,\psi,s):=L(V_f\otimes V_\psi,s)=L(\pi_{f}\times\pi_\psi, s-\frac{k-1+\ell_1+\ell_2}{2}),$$ where $\pi_f, \pi_\psi$ are the unitary automorphic representations of $\GL_2(\A_\Q)$ attached to $f$ and $\psi$ respectively. It is defined as an Euler product, and can be completed to a meromorphic function  $$\Lambda(f,\psi,s)=L_\infty(f,\psi,s)L(f,\psi,s),$$ that is an entire function if $\chi_f\cdot\chi_K\cdot\varepsilon_\psi\neq1$. Moreover, it satisfies a functional equation of the form 
\begin{equation}\label{eq: functional equation f,psi}
\Lambda(f,\psi,s)=\epsilon(f,\psi)\Lambda(\bar{f},\bar{\psi},k+\ell_1+\ell_2-s).
\end{equation}
Following the terminology of \cite{BDP1}, the character $\psi$ is said to be \textit{central critical} for $f$ if $\Lambda(f,\psi^{-1},s)$ is selfdual, $s=0$ is the center of symmetry in the functional equation, and the factor $L_\infty(f,\psi^{-1},s)$ has no poles at $s=0$. Let $\Sigma$ be the set of central critical characters for $f$. Each $\psi\in\Sigma$ satisfies  $\ell_1+\ell_2=k$ and $\varepsilon_\psi=\chi_f$. In particular, the set $\Sigma$ decomposes as
$$\Sigma=\Sigma^{(1)}\sqcup\Sigma^{(2)}\sqcup\Sigma^{(2')},$$ where 
\begin{itemize}
	\item $\Sigma^{(1)}:=\{\psi\in\Sigma\mid 1\leq \ell_1\leq k-1 \text{ and } 1\leq \ell_2\leq k-1  \}$;
	\item $\Sigma^{(2)}:=\{\psi\in\Sigma\mid \ell_1\geq k, \ell_2\leq 0  \}$;
	\item $\Sigma^{(2')}:=\{\psi\in\Sigma\mid \ell_1\leq0, \ell_2\geq k  \}$.
\end{itemize}

Assume that the pair $(N_f,K)$ satisfies Heegner Hypothesis \ref{ass: HH} and let $\cN_f$ be an ideal of $\cO_K$ of norm $N_f$. We now recall the definition of  $p$-adic $L$-function attached to $f$ and $K$ constructed in \cite{BDP1} that interpolates central critical  values $L(f,\psi^{-1},0)$.

We will denote $\Sigma^{(i)}(c,\cN_f,\chi_f)$ the elements of $\Sigma^{(i)}$ of finite type $(c,\cN_f,\chi_f)$.
In \cite{BDP1} the $p$-adic $L$-function $$\L_p(f,K):\hat{\Sigma}(c,\cN_f,\chi_f)^{(2)}\longrightarrow\C_p$$ is defined on the completion of $\Sigma(c,\cN_f,\chi_f)^{(2)}$ with respect to an adequate $p$-adic topology, and it is characterised by the following interpolation property.
\begin{proposition}\label{prop: BDP interpolation formula}
	For each $\psi\in\Sigma(c,\cN_f,\chi_f)^{(2)}$ of infinity type $(k+j,-j)$ with  $j\geq0$,
	\begin{equation*}\label{eq BDP interpolation formula}
	\L_p(f,K)(\psi)=\Big(\frac{\Omega_p}{\Omega} \Big)^{2(k+2j)}\fe(f,\psi)^2\fa(f,\psi)\ff(f,\psi) L(f,\psi^{-1},0)
	\end{equation*}
	where
	\begin{enumerate}[\indent $i)$]
		\item $\Omega$ (resp. $\Omega_p$) is the complex (resp. $p$-adic) period determined by $A_c$ defined in \cite[(5.1.15)]{BDP1} (resp. \cite[(5.2.2)]{BDP1});
		\item $\fe(f,\psi)=1-\psi^{-1}(\bar{\p})a_p(f)+\psi^{-2}(\bar{\p})\chi_f(p)p^{k-1}$;
		\item $\fa(f,\psi)=\pi^{k+2j-1}j!(k+j-1)!$;
		\item $\ff(f,\psi)=\dfrac{2^{k+2j-2} }{(c\sqrt{D_K})^{k+2j-1}}\prod_{q\mid c}\dfrac{q-\chi_K(q)}{q-1}\omega(f_k,\psi)^{-1}\#(\cO_K^\times)$  where $\omega(f,\psi)$ is the scalar of complex norm $1$ defined in \cite[(5.1.11)]{BDP1}.
	\end{enumerate}
\end{proposition}
The set $\Sigma(c,\cN_f,\chi_f)^{(1)}$ is contained in the completion $\hat{\Sigma}(c,\cN_f,\chi_f)^{(2)}$, and the main theorem of \cite{BDP1}  (and its extension of \cite{BDP2}) relates the values of the $p$-adic $L$-function at characters in $\Sigma(c,\cN_f,\chi_f)^{(1)}$ to the generalised Heegner cycles.
As in \S\ref{sec: Generalised Heegner cycles}, we fix an elliptic curve $A$ with complex multiplication by $\cO_K$, defined over the Hilbert class field $K_1$ of $K$. Assume that $A$ has good reduction at $p$ and let $r\leq k-2$ be an integer such that $k\equiv r \mod2$.
Let $H$ be a number field over which all the structures above are defined, let $\omega_A$  be a generator of $\Omega^1(A/H)$ and, considering the algebraic splitting 
$$\coh^1_{\derham}(A/H)=\Omega^1(A/H)\oplus\coh^{0,1}_{\derham}(A/H),$$ let  $\eta_A\in\coh^1_{\derham}(A/H)$ be the element such that $\langle\omega_A,\eta_A\rangle=1$.
For each $j\in\{0,\dots,r \}$, we define 
$$\omega_A^j\eta_A^{r-j}:= \epsilon_{A^{r}}^*(p_1^*\omega_A\wedge\cdots\wedge p_j^*\omega_A\wedge p_{j+1}^*\eta_A\wedge\cdots\wedge p_r^*\eta_A)$$ where  $p_1,\dots,p_r:A^{r}\rightarrow A$ are the projections.
The set $\{\omega_A^j\eta_A^{k-2-j} \mid j=0,\dots k-2 \}$ forms a basis for $\mathrm{Sym}^{k-2}\coh^1_\derham(A/H)$. The following theorem is due to Bertolini--Darmon--Prasanna and Castell\`{a}.
\begin{theorem}[Bertolini--Darmon--Prasanna, Castell\`{a}]\label{thm BDP main thm1}
	For each Hecke character $\psi\in\Sigma(c,\cN_f,\chi_f)^{(1)}$ of infinity type $(r-j,j)$ with $0\leq j\leq r$,
	\begin{alignat*}{2}
	\L_p(f,K)(\psi \norm_K^\frac{k-r}{2})= & \ \fe(f,\psi \norm_K^\frac{k-2-r}{2})^2\frac{\Omega_p^{r-2j}}{(j+1)!\cdot c^{2j}\cdot (4d_c)^\frac{k-2-r}{2}}\Big( \AJ_p(\Delta_{k-2,r,c}^\psi)(\omega_f\wedge\omega_A^j\eta_A^{r-j})\Big)^2,
	\end{alignat*}
	where $\omega_f\in\fil^{k-1}\epsilon_{W_{k-2}}\coh_\derham^{k-1}(W_{k-2}/H)$ is the differential attached to $f$ defined in \S\ref{sec: Modular forms and the cohomology of Kuga--Sato varieties} and $\norm_K = \norm\circ \norm_{K/\Q}$ is the norm character on $K$.
\end{theorem}
The above result is  \cite[Theorem 4.2.5]{BDP2} when $W_{k-2}$ has good reduction at $p$ (in loc.\ cit. only the case $c=1$ is treated, but the proofs therein generalize to $c>1$ with $(c,ND)=1$). The formula was extended by Castell\`{a} to the case of semistable reduction \cite[Theorem 2.11]{Cas18} (again, Castell\`{a} works with trivial character but the proofs extend to the case of non-trivial character).

Castell\`{a} constructed in \cite{Castella-variation} a two variable $p$-adic $L$-function that interpolates the square roots of $\L_p(f,K)(\psi)$ for forms $f$ varying in a Hida family. More precisely, if $\bff\in\Lambda_\bff[[q]]$ is a Hida family Castell\`{a}'s construction provides a two variable function $\L_p(\bff,K)(k,\psi)$ (defined on an appropriate weight space $\cW$) such that for $(k,\psi)\in\cW^{\mathrm{cl}}$ one has that
\begin{align*}
  \L_p(\bff,K)(k, \psi )^2 = \L_p(f_k,K)(\psi).
\end{align*}
The factorization formula that we will prove in \ref{sec: a factorization formula} involves Castell\`{a}'s $p$-adic $L$-function evaluated at characters which are the classical specializations of Hida families of theta series, which we next recall.

\subsubsection{Hida families of  theta series}\label{sec: Hida families of Theta series}
Let $\psi_g$ be a Hecke character of the imaginary quadratic field $K$ of infinity type $(0,\ell_0-1)$ and conductor $\fc$.
Define  $g:=\theta(\psi_g)\in S_{\ell_0}(N_g,\chi_g)$ to be the theta series attached to $\psi_g$. 
Fix a prime number $p$ not dividing $N_g$ and assume that it splits in $K$ as $p\cO_K=\p\bar{\p}$. There is a Hida family $\bfg$ of theta series passing through $g_\alpha$, whose construction can be found in \cite[\S 5]{Gh05}. We next recall the specialisations at integer weights of this family.

Fix a Hecke character $\lambda$ of $K$ with infinity type $(0,1)$ and conductor $\bar{\p}$. Let $\Q_p(\lambda)$ be the field obtained by adjoining to $\Q$ the values of $\lambda$ and taking the $p$-adic completion. Consider the factorisation of its group of units  $\cO_{\Q_p(\lambda)}^\times=\mu\times W$ where $\mu$ is finite and $W$ is free over $\Z_p$, and take the projection to the second factor 
\begin{align*}
  \langle\cdot\rangle:\cO_{\Q_p(\lambda)}^\times\longrightarrow W.
\end{align*}
For each $\ell\in\Z_{\geq0}$ such that $\ell\equiv\ell_0\pmod{p-1}$, define 
\begin{align*}
\psi_{g,\ell-1}^{(p)}:=\psi_g\langle\lambda\rangle^{\ell-\ell_0}\ \  \text{ and }\ \  \psi_{g,\ell-1}(\fq):=\begin{cases}
\psi_{g,\ell-1}^{(p)}(\fq) & \fq\neq\bar{\fq};\\
\chi_g(p)p^{\ell-1}/\psi_{g,\ell-1}^{(p)}(\p) & \fq=\bar{\p}. 
\end{cases}
\end{align*}
Then $\psi_{g,\ell-1}$ is a Hecke character of infinity type $(0,\ell-1)$. Define 
$$g_\ell:=\theta(\psi_{g,\ell-1})\in M_\ell(N_g,\chi_g)$$ 
The $p$-stabilisation of this modular form is the theta series $$\bfg_\ell=\theta(\psi_{g,\ell-1}^{(p)})\in S_\ell(N_gp,\chi_g).$$

\subsubsection{Katz's $p$-adic $L$-function}

Using the notation of the previous sections, fix an ideal $\fc$ of the ring $\cO_K$ and let $\Sigma(\fc)$ be the set of Hecke characters of $K$ with conductor dividing $\fc$.
Let $\Sigma_K^{(1)}(\fc)$ and $\Sigma_K^{(2)}(\fc)$ be the subsets of $\Sigma_K(\fc)$ containing the characters of infinity type $(\ell_1,\ell_2)$ such that $\ell_1\leq0,\ell_2\geq1$ and $\ell_1\geq1,\ell_2\leq0$ respectively. Define $$\Sigma_K(\fc):=\Sigma_K^{(1)}(\fc)\sqcup\Sigma_K^{(2)}(\fc).$$
Then for each $\psi\in\Sigma_K(\fc)$ the point $s=0$ is central critical for the complex $L$-function $L(\psi^{-1},s)$.

In \cite{Kat76}, Katz attached to $K$ a $p$-adic $L$-function $\L_p(K)$, defined on the completion of $\Sigma_K^{(2)}(\fc)$ with respect an adequate $p$-adic topology and characterised by the following interpolation property.

\begin{proposition}\label{prop: interpolation property for katz padic Lfunction}
	For each $\psi\in\Sigma_K^{(2)}(\fc)$ with infinity type $(\ell_1,\ell_2)$, we have
	\begin{equation}
	\L_p(K)(\psi)=\fa(\psi)\fe(\psi)\ff(\psi) \Big( \frac{\Omega_p}{\Omega}\Big)^{\ell_1-\ell_2}L_\fc(\psi^{-1},0),
	\end{equation}
	where
	\begin{enumerate}[\indent $i)$]
		\item $L_\fc(\psi^{-1},s)$ is the product of all the Euler factors defining $L(\psi^{-1},s)$ except the ones corresponding to the primes dividing $\fc$; 
		\item $\fa(\psi)=(\ell_1-1)!\pi^{-\ell_2}$;
		\item $\fe(\psi)=(1-\psi(\p)p^{-1})(1-\psi^{-1}(\bar{\p}))$;
		\item $\ff(\psi)=(D_K)^{\ell_2/2}2^{-\ell_2}$.
		\item $\Omega$ and $\Omega_p$ are the periods appearing in Proposition \ref{prop: BDP interpolation formula}.
	\end{enumerate}
\end{proposition}

From the functional equation satisfied by complex Hecke $L$-functions, it follows a functional equation for Katz $p$-adic $L$-function. In particular, if $\psi$ is a finite order character such that $(\psi')^{-1}=\psi$, then
\begin{equation}\label{eq: functional equation Katz padic Lfuncion}
\L_p(K)(\psi)=\L_p(K)(\psi\norm_K).
\end{equation}
Moreover, in \cite[\S10.4.9, \S10.4.12]{Kat76}, Katz related the values of $\L_p(K)$ at finite order characters to elliptic units.
\begin{theorem}\label{thm: katz padic Lfunction and elliptic units}
	Let $\psi$ a finite order character of $K$ of conductor $\fc$ and let $c$ be the smallest positive integer in $\fc$. Then
	$$\L_p(K)(\psi)=\begin{cases}
	\frac{1}{2}(\frac{1}{p}-1)\log_p(u_\p) & \text{ if } \psi=1;\\
	-\frac{1}{24c}\fe(\psi)\sum_{\sigma\in\Gal(K_c/K)}\psi^{-1}(\sigma)\log_p(\sigma(u)) &  \text{ if } \psi\neq1;
	\end{cases}$$
	where $u_\p\in K^\times$ is a generator of the principal ideal $\p^{h_K}$ and $u\in\cO_{K_c}^\times$ is an elliptic unit.
\end{theorem}

\subsection{A factorisation formula and the proof of the special case}\label{sec: a factorization formula} We resume the notations and assumptions described in the beginning of \S \ref{sec: statement of the particular case}. We also let let $\bfg,\bfh$ be the Hida families of theta series passing through  $g_\alpha=\theta(\psi_g^{(p)})$ and $h_\alpha=\theta(\psi_h^{(p)})$ as described in \S\ref{sec: Hida families of Theta series}. Recall that, using the notation of \S\ref{sec: Hida families of Theta series}, the specialisation of $\bfg$ at a point of weight $\ell$ is the $p$-stabilisation of $g_\ell:=\theta(\psi_{g,\ell-1})$, and similarly for $\bfh$.

Define for each $k,\ell,m\in\Z_{\geq1}$ the following Hecke characters: 
\begin{enumerate}[\indent $i)$]
	\item $\Psi_g(\ell):=\psi_{g,\ell-1}^{-2}\chi_g\norm_K^\ell$;
	\item $\Psi_{gh}(k,\ell,m):=(\psi_{g,\ell-1}\psi_{h,m-1})^{-1}  \norm_K^{\frac{k+\ell+m-2}{2}}$;
	\item $\Psi_{gh'}(k,\ell,m):=(\psi_{g,\ell-1}\psi'_{h,m-1})^{-1} \norm_K^{\frac{k+\ell+m-2}{2}}$.
\end{enumerate}

For each $k,\ell,m\in\Z_{\geq1}$, we have the following decomposition of the triple tensor product of representations
\begin{equation}\label{eq the triple product representation splits}
	V_{f_k}\otimes V_{g_\ell}\otimes V_{h_m}=V_{f_k}\otimes V_{\psi_{g,\ell-1}}\otimes V_{\psi_{h,m-1}}=V_{f_k}\otimes V_{\psi_{g,\ell-1}\psi_{h,m-1}}\oplus V_{f_k}\otimes V_{\psi_{g,\ell-1}\psi'_{h,m-1}}.
\end{equation}

This induces a factorisation of complex $L$-functions, up to a finite number of factors at the bad reduction primes.  Evaluating at the central critical point $c_0:=\frac{k+\ell+m-2}{2}$ we obtain a factorization of the form
\begin{equation}\label{eq: factorisation of complex Lfunctions}
\begin{split}
L( f_k\otimes g_\ell \otimes h_m,c_0) & 	= \ff_1(f_k,g_\ell,h_m)L(f_k,\psi_{g,\ell-1}\psi_{h,m-1},c_0)L(f_k,\psi_{g,\ell-1}\psi'_{h,m-1},c_0)\\
& =  \ff_1(f_k,g_\ell,h_m)L(f_k,\Psi_{gh}(k,\ell,m)^{-1},0)L(f_k,\Psi_{gh'}(k,\ell,m)^{-1},0),
\end{split}
\end{equation} 
where $\ff_1(f_k,g_\ell,h_m)$ accounts for the evaluation of the Euler factors at bad reduction primes. From this decomposition it follows a factorisation of the triple $p$-adic $L$-function in terms of Katz's and Castell\`{a}'s $p$-adic $L$ functions.

\begin{theorem}\label{thm: factorisation of padic Lfunctions}
	For each $(k,\ell,m)\in\cW_{\bff\bfg\bfh}^{\circ}$  we have
	$$\L_p^g(\ubff,\ubfg,\ubfh)^2(k,\ell,m)\L_p(K)(\Psi_g(\ell))^2=\L_p(K,\bff)(k,\Psi_{gh}(k,\ell,m))^2\L_p(K,\bff)(k,\Psi_{gh'}(k,\ell,m))^2\ff(k,\ell,m),$$
	where	
	\begin{enumerate}[\indent $i)$]
		\item $\ff(k,\ell,m):=\dfrac{\prod_{q\in\Sigma_\mathrm{exc}}(1+q^{-1})}{(-4)^{\ell-2}} \dfrac{\ff(\Psi_g(\ell))^2\ff_1(f_k,g_\ell,h_m)}{\ff_2(g_\ell)^2\ff_3(g_\ell)^2\ff(k,\Psi_{gh}(k,\ell,m))\ff(k,\Psi_{gh'}(k,\ell,m))}$;
		\item $\ff(\Psi_g(\ell))$ is the factor appearing in Proposition \ref{prop: interpolation property for katz padic Lfunction};
		\item $\ff_1(f_k,g_\ell,h_m)$ is the factor appearing in (\ref{eq: factorisation of complex Lfunctions});
		\item $\ff(k,\Psi_{gh}(k,\ell,m))$ and $\ff(k,\Psi_{gh'}(k,\ell,m))$ are the factors appearing in Proposition \ref{prop: BDP interpolation formula}; 
		\item $\ff_2(g_\ell)$ is the factor defined by the equality $L(\Psi_g(\ell),0)=\ff_2(g_\ell)L_\fc(\Psi_g(\ell),0)$, where $\fc:=\mathrm{lcm}(\fc_g,\fc_h)$;
		\item $\ff_3(g_\ell)$ is the factor defined by the equality $\langle g_\ell^*,  g_\ell^*\rangle=(\ell-1)!\pi^{-\ell}\ff_3(g_\ell)L(\Psi_g(\ell),0)$ of \cite[Lemma 3.7]{DLR}, where $g_\ell^*:=g_\ell\otimes\chi_g^{-1}$.
	\end{enumerate}
\end{theorem}

\begin{proof}
	By Theorem \ref{thm: Hsieh interpolation formula}, for each  $(k,\ell,m)\in\cW_{\bff\bfg\bfh}^{g}$ we have
	\begin{alignat*}{2}
	L(f_k\otimes g_\ell\otimes h_m,c_0) = & \L^g_p{(\ubff,\ubfg,\ubfh)}(k,\ell,m)^2\dfrac{(-4)^\ell\langle g_\ell,g_\ell\rangle^2\mathcal{E}_0(g_\ell)\mathcal{E}_1(g_\ell)}{\mathcal{E}(f_k,g_\ell,h_m)^2}\times\\
	& \times\dfrac{1}{\fa(k,\ell,m)}\dfrac{1}{\prod_{q\in\Sigma_\mathrm{exc}}(1+q^{-1})}.
	\end{alignat*}
	Let $c$ be the smallest positive integer in $\fc$, then if $(x,y,z)\in\cW_{\bff\bfg\bfh}^{g}$, the characters $\Psi_{gh}(k,\ell,m)$ and $\Psi_{gh'}(k,\ell,m)$ belong to $\Sigma(c,\cN_f,\chi_f)^{(2)}$. Indeed, $L(f_k,\Psi_{gh}(k,\ell,m)^{-1},0)=L(f_k,\psi_{g,\ell-1}\psi_{h,m-1},c_0)$, and $c_0$ is central critical for this complex $L$-function. Moreover, $\Psi_{gh}(k,\ell,m)$ has infinity type 
	\begin{align}\label{eq: infty type 1}
\Big(\frac{k+\ell+m-2}{2}, \frac{k-\ell-m+2}{2} \Big)=(k+j,-j)
\end{align}
 with $j=\frac{-k+\ell+m-2}{2}\geq0$.
	Similarly, $\Psi_{gh'}(k,\ell,m)$ has infinity type 
\begin{align}\label{eq: infty type 2}
\Big(\frac{k+\ell-m}{2}, \frac{k-\ell+m}{2} \Big)=(k+j,-j)
\end{align}
 with $j=\frac{\ell-k-m}{2}\geq0$.
	Then using (\ref{eq: factorisation of complex Lfunctions}) and Proposition \ref{prop: BDP interpolation formula} we obtain 
	\begin{equation}\label{eq: 1}
	\begin{split}
	& \L^g_p{(\ubff,\ubfg,\ubfh)}^2(x,y,z)\big(\dfrac{\Omega}{\Omega_p}\big)^{4-4\ell}\langle g_\ell,g_\ell\rangle^2  \\ & =\L_p(K,\bff)(k,\Psi_{gh}(k,\ell,m))^2\L_p(K,\bff)(k,\Psi_{gh'}(k,\ell,m))^2\times\dfrac{\prod_{q\in\Sigma_\mathrm{exc}}(1+q^{-1})\ff_1(f_k,g_\ell,h_m)}{(-4)^\ell\ff(k,\Psi_{gh}(k,\ell,m))\ff(k,\Psi_{gh'}(k,\ell,m))}\\ 
	& \times \dfrac{\fa(k,\ell,m)}{\fa(\Psi_{gh}(k,\ell,m))\fa(\Psi_{gh'}(k,\ell,m))}
	 \times\dfrac{\mathcal{E}(f_k,g_\ell,h_m)^2}{\mathcal{E}_0(g_\ell)^2\mathcal{E}_1(g_\ell)^2\fe(k,\Psi_{gh}(k,\ell,m))^2\fe(k,\Psi_{gh'}(k,\ell,m))^2}.
	\end{split}
	\end{equation}
	
	On the other hand, the character $\Psi_g(\ell)$ has infinite type $(\ell,2-\ell)$ and conductor dividing $\fc$, so for $\ell\geq2$ it belongs to $\Sigma(\fc)^{(2)}$. Substituting \cite[(53) and Lemma 3.7]{DLR} in the interpolation formula of Proposition  \ref{prop: interpolation property for katz padic Lfunction}, we obtain, for each $\ell\geq2$, 
	\begin{equation}\label{eq:2}
	L_p(K)(\Psi_g(\ell)) = \fa(\Psi_g(\ell))\fe(\Psi_g(\ell))\dfrac{\ff(\Psi_g(\ell))}{\ff_2(\ell)\ff_3(\ell)}\Big( \frac{\Omega_p}{\Omega} \Big)^{2\ell-2}\langle g_\ell,g_\ell\rangle\frac{\pi^\ell}{(\ell-1)!}. 
	\end{equation}
	Plugging (\ref{eq:2}) into (\ref{eq: 1}) it follows that:
	\begin{equation*}
	\begin{split}
&	\L^g_p(\ubff,\ubfg,\ubfh)(x,y,z)^2\L_p(K)(\Psi_g(\ell))^2=  \L_p(K,\bff)(k,\Psi_{gh}(k,\ell,m))^2\L_p(K,\bff)(k,\Psi_{gh'}(k,\ell,m))^2\\
	& \times\dfrac{\ff(k,\ell,m)\pi^{2\ell}\fa(\Psi_g(\ell))^2\fa(f_k,g_\ell,h_m)}{2^4[(\ell-1)!]^2\fa(\Psi_{gh}(k,\ell,m))\fa(\Psi_{gh'}(k,\ell,m))}
	 \times\dfrac{\fe(\Psi_g(\ell))^2\mathcal{E}(f_k,g_\ell,h_m)^2}{\mathcal{E}_0(g_\ell)^2\mathcal{E}_1(g_\ell)^2\fe(k,\Psi_{gh}(k,\ell,m))^2\fe(k,\Psi_{gh'}(k,\ell,m))^2}.
	\end{split}
	\end{equation*}
	Then the statement of the theorem follows from the identities
	\begin{enumerate}[\indent $i)$]
		\item $\dfrac{\pi^{2\ell}\fa(\Psi_g(\ell))^2\fa(f_k,g_\ell,h_m)}{[(\ell-1)!]^2\fa(\Psi_{gh}(k,\ell,m))\fa(\Psi_{gh'}(k,\ell,m))}=2^4$,
		\item $\mathcal{E}(f_k,g_\ell,h_m)=\fe(k,\Psi_{gh}(k,\ell,m))\fe(k,\Psi_{gh'}(k,\ell,m))$,
		\item $\fe(\Psi_g(\ell))=\mathcal{E}_0(g_\ell)\mathcal{E}_1(g_\ell)$,
	\end{enumerate}
	and by continuity.
\end{proof}
As we will see, the proof of Theorem \ref{thm: special case weight (k,1,1)} follows from evaluating the formula of Theorem \ref{thm: factorisation of padic Lfunctions} at weights $(k,1,1)$. Since it will be needed, we record the following property on the field of definition of $\ff(k,1,1)$.
\begin{proposition}\label{prop: f is admissible}
 If $k$ is even then $\ff(k,1,1)\mathfrak{g}(\chi_f)^2$ belongs to $ L$.
\end{proposition}
\begin{proof}
It follows readily from the definitions that the several factors that enter into the definition of $\ff(k,1,1)$ belong to $L$, except the factors $\ff(k,\Psi_{gh}(k,1,1))$ and $\ff(k,\Psi_{gh'}(k,1,1))$. Indeed, these factors are defined in terms of certain scalars $\omega(f_k,\Psi_{gh}(k,1,1))$ and $\omega(f_k,\Psi_{gh'}(k,1,1))$. By \cite[(5.1.11)]{BDP1}, we have that
\begin{align*}
\omega(f_k,\Psi_{gh}(k,1,1))\omega(f_k,\Psi_{gh'}(k,1,1))=\dfrac{\omega_f^2(-N_f)^{\ell-1}\norm_{K/\Q}(\mathfrak{b})^{\ell-1}}{\chi_f(\norm_{K/\Q}(\mathfrak{b}))^2\psi_g(\mathfrak{b})^2{\psi_h(\mathfrak{b})\psi'_h(\mathfrak{b})}b^{2\ell+2}}.
\end{align*}

Here $\mathfrak{b}$ is an choice of an ideal of $\cO_c$ prime to $pN_fc$ and $\mathfrak{b}\cdot\cN_f=(b)$, and $w_f$ is the scalar such that $W_{N_f}  f^*_k = w_f f_k$ (here $W_{N_f}$ is the Atkin--Lehner involution).  The statement then follows from \cite[Theorem 2.1]{atkin-li}, which implies, when $k$ is even, that $w_f\mathfrak{g}(\chi_f)$ belongs to $L$.

\end{proof}

\subsubsection{Proof of Theorem \ref{thm: special case weight (k,1,1)}}\label{sec: proof of the special case} In this paragraph we will use the notation and assume all the hypothesis of  \S\ref{sec: statement of the particular case}. In particular, $f$ is a normalised cuspidal newform of weight $k\geq2$ and $g,h$ are theta series of the finite order Hecke characters $\psi_g,\psi_h$ of the imaginary quadratic field $K$ in which the prime number $p$ splits.
Let $\bff$ the Hida family passing through the only ordinary $p$-stabilisation $f_\alpha$ of $f$, let $\bfg,\bfh$ be the Hida families of  theta series of \S\ref{sec: Hida families of Theta series} passing through $g_\alpha$ and $h_\alpha$ respectively and let $(\ubff,\ubfg,\ubfh)$ be the choice of test vectors of Theorem \ref{thm: Hsieh interpolation formula}. Then, evaluating the factorisation formula of Theorem \ref{thm: factorisation of padic Lfunctions} at $(k,1,1)$ and taking square roots we obtain:
$$\L_p^g(\ubff,\ubfg,\ubfh)(k,\ell,m)\L_p(K)(\Psi_g(1))=\L_p(K,\bff)(k,\Psi_{gh}(k,1,1))\L_p(K,\bff)(k,\Psi_{gh'}(k,1,1))\ff'(k,1,1),$$
where $\ff'(k,1,1):=\sqrt{\ff(k,1,1)}$.
Then the statement of Theorem  \ref{thm: special case weight (k,1,1)} follows applying Theorem \ref{thm: katz padic Lfunction and elliptic units} and Theorem \ref{thm BDP main thm1}, and Proposition \ref{prop: f is admissible}, after observing that: 
\begin{enumerate}[\indent $i)$]
	\item $\Psi_g(1)=\psi\norm_K$, where $\psi:=\psi_g'/\psi_g$ has finite order and it is selfdual, so that, by (\ref{eq: functional equation Katz padic Lfuncion}, $\L_p(K)(\Psi_g(1))=\L_p(K)(\psi)$.
	\item $\Psi_{gh}(k,1,1)=(\psi_g\psi_h)^{-1}\norm_K^{k/2} =\psi_1^{-1}\norm_K^\frac{k-r}{2}$ and $\psi_1^{-1}=(\psi_g\psi_h)^{-1}$ has infinity type $(r-j,j)$ with $r:=0$, $j:=0$, and analogously for $\Psi_{gh'}(k,1,1)$.
\end{enumerate}

\section{ The conjecture for general unbalanced weights}\label{sec: general unbalanced weights}
Let $(f,g,h)$ be a triple of normalised newforms of levels $(N_f, N_g, N_h)$ and weights $(k,\ell,m)$ with $k\geq\ell+m$ and $k,\ell,m\geq2$, and fix a prime number $p$ such that $p\nmid N_gN_h$ and $\ord_p(N_f)\leq1$ and assume that $f,g$ and $h$ are ordinary at $p$. The aim of this section is to formulate a version of the Elliptic Stark Conjecture for $(f,g,h)$ in this setting, and to give some theoretical evidence for the conjecture in a special case. More precisely, in \S\ref{sec: statement of the conjecture in more general weights}, we state the conjecture, and in \S\ref{sec: statement of the particular case higher general weights} we focus on the case in which $g$ and $h$ are theta series of an imaginary quadratic field where the prime $p$ splits. In this setting, we prove a formula relating the value $\L_p^g(\bff,\bfg,\bfh)(k,\ell,m)$ to the $p$-adic Abel--Jacobi image of certain generalised Heegner cycles, using the factorisation of Theorem \ref{thm: factorisation of padic Lfunctions} and the result of \cite{BDP2} that we stated as Theorem \ref{thm BDP main thm1}. Finally, in \S\ref{subsection: The proof of a special case} we particularize the formula to a triple of forms $(f,g,h)$ satisfying some additional hypothesis in order to obtain a proof of the conjecture for such triple, conditional on the validity of Tate's conjecture for motives and of certain standard conjectures on the $p$-adic Abel--Jacobi map.

\subsection{Statement of the conjecture}\label{sec: statement of the conjecture in more general weights} We begin by recalling some notation and terminology related to motives.  We refer to \cite{scholl} for further details. For two number fields $K$ and $F$, denote by $\cM(K)_F$ the category of Chow motives over $K$ with coefficients in $F$. The objects of $\cM(K)_F$ are triples $(V,q,m)$ where $V$ is a smooth projective scheme over $K$, $q=q^2$ is a projector in the ring of correspondences of $V$ tensored with $F$, and $m$ is an integer. 
For $i=1,2$, let $M_i:=(V_i,q_i,m_i)$  be an object of $\cM(K)_F$ and assume that $V_1$ is purely of dimension $d_1$; the morphisms from $M_1$ to $M_2$ are defined in terms of correspondences between the underlying varieties as
\begin{equation*}
\hom(M_1,M_2):=q_1\circ\operatorname{Corr}^{m_2-m_1}(V_1,V_2)\circ q_2,
\end{equation*}
where $\operatorname{Corr}^{m_2-m_1}(V_1,V_2):=\CH^{d_1+m_2-m_1}(V_1\times V_2)\otimes_\Z F$.
Let  $\mathbb{L}:=(\operatorname{Spec}(K),\operatorname{id},-1)$ be the Lefschetz motive and let $d$ be an integer. We denote  $\mathbb{L}^d:=\mathbb{L}^{\otimes^{d}}$ the tensor product of $\mathbb{L}$ with itself $d$ times. The \textit{Chow group} of a motive $M\in \cM(K)_F$ is defined as 
\begin{equation}\label{eq: chow group of motives}
\CH^d(M):=\hom(\mathbb{L}^{d},M).
\end{equation}

The Chow group of a motive can also be interpreted as a group of cycles, since  
\begin{equation*}
\CH^d((V,q,m))\cong q\cdot\CH^{d+m}(V/K)_F.
\end{equation*}
Then $\CH^d(M)_0$ is defined as the subgroup of the null-homologous cycles of $\CH^d(M)$.  We will occasionally use the notation $\CH^d(M)_{0,F}$ if we need to emphasize the field of coefficients of the Chow group.

Let $(f,g,h)$ be a triple of forms of weights $(k,\ell,m)$ with $k\geq\ell+m$ and $k,\ell,m\geq2$, and let $L$ be a number field that contains the Fourier coefficients of $f$, $g$, and $h$.
The motive attached to $f\otimes g\otimes h$ is the object of $\cM(\Q)_L$ obtained as the tensor product of motives attached to $f$, $g$, and $h$: $$M(f\otimes g \otimes h):=M_f\otimes M_g\otimes M_h,$$ whose underlying variety  is 
\begin{align*}
X := W_{k-2}\times W_{\ell-2}\times W_{m-2}.
\end{align*}

 Put $c:= (k+\ell+m-2)/2$ and suppose that    $$\dim_L\CH^c(M(f\otimes g \otimes h))_{0,L}=2.$$ Under this assumption, we can define the following regulator attached to $(f,g,h)$. 

\begin{definition}
Let $ \Delta_1,\Delta_2 $ be a basis of $\CH^c(M(f\otimes g \otimes h))_{0,L}$.  The \textit{regulator} attached to $(f,g,h)$ is
\begin{align}\label{eq regulator higher weights}
\operatorname{Reg}(f,g,h):=\left|
\begin{array}{rr}
\AJ_p(\Delta_1)(\omega_f\wedge \eta_g\wedge\omega_h) &  \AJ_p(\Delta_1)(\omega_f\wedge \eta_g\wedge\eta_h)\\
\AJ_p(\Delta_2)(\omega_f\wedge \eta_g\wedge\omega_h) &  \AJ_p(\Delta_2)(\omega_f\wedge \eta_g\wedge\eta_h)
\end{array}
\right|,
\end{align}
where $\omega_f,\eta_g,\eta_h,\omega_h$ are the de Rham classes defined in \S\ref{sec: Modular forms and the cohomology of Kuga--Sato varieties}.
\end{definition}

\begin{remark}
  Since the definition of regulator involves the choice of an $L$-basis of $\CH^c(M(f\otimes g \otimes h))_{0,L}$, it is only defined up to multiplication by an element of $L^\times$. 

Denote $\bff,\bfg,\bfh$ the Hida families passing through the ordinary $p$-stabilisations of $f,g$ and $h$. The following is the analog of the Elliptic Stark Conjecture in this setting.
\end{remark}
\begin{conjecture}\label{conj: more general weights} Set $r: = \dim_L\CH^c(M(f\otimes g \otimes h))_{0,L}$.
	\begin{enumerate}[\indent $i)$]
		\item If $r> 2$, then $\L^g_p(\ubff,\ubfg,\ubfh)(k,\ell,m)=0$ for any choice of test vectors $(\ubff,\ubfg,\ubfh)$ for $(\bff,\bfg,\bfh)$;
		\item if $\ord_{s=c}L(f\otimes g\otimes h,s)=2$, then  there exists a finite extension $L_0$ of $L$, a triple of test vectors
		\begin{align*}
		(\breve f, \breve g_\alpha, \breve h)\in S_k(Np, \chi_f)_L[f]\times M_\ell(Np,\chi_g)_L [g_\alpha]\times M_m(Np,\chi_h)_L[h]  
		\end{align*}
		and Hida families $\ubff,\ubfg,\ubfh$ with $f_k = \breve f$, $g_\ell = \breve g$, $h_m = \breve h$  such that
		\begin{equation}\label{eq: conjecture (k,l,m)}
	\L^g_p(\ubff,\ubfg,\ubfh)(k,\ell,m)=\mathrm{Reg}(f,g,h) \mod L_0^\times.
		\end{equation} 
	\end{enumerate}
\end{conjecture}

\subsection{The case of theta series of imaginary quadratic fields}\label{sec: statement of the particular case higher general weights}  Let $K$ be an imaginary quadratic field of discriminant coprime to $N_f$ in which the prime $p$ splits as $p\cO_K=\fp\bar{\fp}$ and such that the pair $(K,N_f)$ satisfies the Heegner Hypothesis (cf. Assumption \ref{ass: HH}). In this subsection we consider the case in which $g$ and $h$ are theta series of two Hecke characters $\psi_g,\psi_h$ of $K$. We will use the same notations and assume the same hypothesis of \S\ref{sec: statement of the particular case}, with only two differences. The first one is that in \S \ref{sec: statement of the particular case} the characters $\psi_g,\psi_h$ where assumed to be of infinity type $(0,0)$, whereas we now suppose that they are of infinity type $(0,\ell-1), (0, m-1)$ for some $\ell,m\geq 2$. The second difference is that now we will define the characters $\psi_1$ and $\psi_2$ to be:
\begin{align}
\psi_1 := \psi_g \psi_h \norm_{K}^{2-\ell-m},\ \ \psi_2 := \psi_g \psi_h' \norm_{K}^{2-\ell-m}.
\end{align}

As in \S\ref{sec: statement of the particular case}, we  assume that for $i=1,2$ the conductor of $\psi_i$ is of the form $c_i\cN_i$ with  $c_i\in\Z$ coprime to $D_KN_{f}$ and $\cN_i\mid \cN_{\chi_f}$. 
Using the notation introduced in \S\ref{sec: Generalised Heegner cycles}, we consider the generalized Heegner cycles 
\begin{align}\label{eq: cycles general weight}
\tilde \Delta^{\psi_i^{-1}}:=\tilde \Delta_{k-2,\ell+m-2,c_i}^{\psi_i^{-1}}\in \CH^{c}(W_{k-2}\times A^{\ell+m-2}/H_{c_i,f})_{0,L},
\end{align}
where $A$ is an elliptic curve defined over the Hilbert class field $K_1$ of $K$ with CM by $\cO_K$ that we fix once and for all.

\begin{proposition}\label{prop: main formula in the general case} 
	Let $(\ubff,\ubfg,\ubfh)$ be the choice of test vector of Theorem \ref{thm: Hsieh interpolation formula}. There exist a quadratic extension $L_0/L$ and $\lambda\in L_0$ such that
	\begin{equation}\label{eq: main formula in the general case}
	\L_p^g(\ubff,\ubfg,\ubfh)(k,\ell,m) =\dfrac{\lambda}{\mu}\cdot \AJ_p(\tilde \Delta^{\psi_1^{-1}})( \omega_f\wedge \eta_A^{\ell + m -2})\cdot \AJ_p(\tilde \Delta^{\psi_2^{-1}})( \omega_f\wedge \eta_A^{\ell-1} \omega_A^{m-1}),
	\end{equation}
	
	where $$\mu:=\Omega^{2-2\ell}\pi^{\ell-2}L(\Psi_g(\ell)^{-1},0)\in\Qbar.$$
\end{proposition}
\begin{proof}
	Let $r:=\ell+m-2$. Applying Theorem \ref{thm BDP main thm1} to the characters
	\begin{enumerate}[$i)$]
		\item $\Psi_{gh}(k,\ell,m)=\psi_1^{-1} \norm_K^\frac{k-r}{2}$ where $\psi_1^{-1}$ has infinity type $(r-j,j)$ with $j=0$;
		\item $\Psi_{gh'}(k,\ell,m)=\psi_2^{-1} \norm_K^\frac{k-r}{2}$ where $\psi_2^{-1}$ has infinity type $(r-j,j)$ with $j=m-1$
	\end{enumerate}
	and substituting into equation (\ref{eq: 1}), we obtain 
	
	$	\L_p^g(\ubff,\ubfg,\ubfh)^2(k,\ell,m) \times {\Omega^{4-4\ell}\langle g, g \rangle^2}\dfrac{\fa(k,\Psi_{gh}(k,\ell,m))\fa(k,\Psi_{gh}(k,\ell,m))}{\fa(k,\ell,m)} = $
	\begin{equation*}
	\begin{split}
	=\dfrac{(-1)^\ell\prod_{q\in\Sigma_\mathrm{exc}}(1+q^{-1})}{m!c_2^{2m-2}4^{k-\ell-m}(d_{c_1}d_{c_2})^\frac{k-\ell-m}{2}}\dfrac{\ff_1(f,g,h)}{\ff(k, \Psi_{gh}(k,\ell,m))\ff(k, \Psi_{gh'}(k,\ell,m))}\dfrac{\cE(f,g,h)^2}{\cE_0(g)^2\cE_1(g)^2}\\
	\times	\AJ_p(\tilde \Delta^{\psi_1^{-1}})^2( \omega_f\wedge \eta_A^{\ell + m -2})\cdot \AJ_p(\tilde \Delta^{\psi_2^{-1}})^2( \omega_f\wedge \eta_A^{\ell-1} \omega_A^{m-1}).
	\end{split}
	\end{equation*}
	Using the definition of the factors involved, the left hand side of the previous equality is 
	$$	\L_p^g(\ubff,\ubfg,\ubfh)^2(k,\ell,m)(\Omega^{2-2\ell}\langle g,g\rangle\pi^{2\ell-2})^2.$$
	As $\ell\geq2$, the character $\Psi_g(\ell)$ belongs to the region of classical interpolation for Katz's $p$-adic $L$-function, and following the computations of the proof of \cite[Lemma 3.7]{DLR}, we obtain 
	$$\langle g, g \rangle=\dfrac{L(\Psi_g(\ell)^{-1},0)}{\pi^\ell \sqrt{D_K}} \mod\Q^\times.$$
	Then we see that $$\mu=(\Omega^{2-2\ell}\langle g,g\rangle\pi^{2\ell-2})^{-1}=(\Omega^{2-2\ell}\pi^{\ell-2}L(\Psi_g(\ell)^{-1},0))^{-1},$$  is algebraic by \cite[Proposition 2.11 (1) and Theorem 2.12]{BDP2}, and the factors $\ff_1(k,\ell,m)$ and $\dfrac{\cE(f,g,h)}{\cE_0(g)\cE_1(g)}$ belong to $L$ by the definition of these factors. 
\end{proof}

\subsection{The proof of a special case}\label{subsection: The proof of a special case} 
The rest of this section will be devoted to analyzing the connection between Proposition \ref{prop: main formula in the general case} and Conjecture \ref{conj: more general weights} in a particular case where $\psi_g$ and $\psi_h$ are powers of the Hecke character of an elliptic curve with CM by $\cO_K$. More precisely, in this subsection we continue to denote by $f$ a modular form of weight $k\geq 2$, level $N_f$ and Nebentype $\chi_f$, and we make the following additional assumptions regarding $K$, $\psi_g$, and $\psi_h$:

\begin{enumerate}
	\item $K$ is an imaginary quadratic field of class number $1$.
	
	\item We fix an elliptic curve $A_0/\Q$ with CM by $K$. We denote by $A := A_0\otimes K$ its extension of scalars to $K$ and by $\psi_A$ the Hecke character of $A$. Then we assume that $\psi_g= \psi_A^{\ell-1}$ and $\psi_h=\psi_A^{m-1}$, with $\ell>m\geq 2$ and  $k\geq \ell+m$. 
\end{enumerate}
As usual, we denote  $g:=\theta(\psi_g)\in S_\ell(N_g,\chi_g)$ and $h:=\theta(\psi_h)\in S_m(N_h,\chi_h)$.
We simplify further the setting taking the following assumption on the discriminant of $K$.	
\begin{assumption}\label{ass: condition on DK}
	The discriminant $-D_K$ of $K$ satisfies one of the following conditions:
	\begin{enumerate}[$i)$]
		\item $D_K$ is odd;
		\item $8\mid D_K$;
		\item there exists a prime $\ell\mid D_K$ such that $\ell\equiv 3 \ \mod (4)$.
	\end{enumerate} 
\end{assumption}

Under this assumption, the elliptic curve $A_0/\Q$ can be constructed as in \cite[\S11]{GrossCM}, so that the conductor of $\psi_A$ is generated by $\sqrt{-D_K}$, a condition that we will assume from now on. From the conditions imposed in this section, and using the fact that $\theta(\psi_A)$ is the cuspform attached to the elliptic curve $A$ that descends to $\Q$, it follows that
\begin{equation}\label{eq: levels special case}
N_g=N_h=D_K^2
\end{equation}
and
\begin{equation}\label{eq: nebentype chars special case}
\chi_g=\chi_K^\ell; \ \ \chi_h=\chi_K^m; \ \ \chi_f=\chi_K^{\ell+m}=\begin{cases}
1 & \text{ if $\ell+m$ even}\\ \chi_K & \text{ if $\ell+m$ is odd.} 
\end{cases}
\end{equation}
In this setting, the involved Hecke characters are

\begin{align}
\psi_1 = \psi_g\psi_h\norm_K^{2-\ell-m}=\psi_A^{\ell+m-2}\norm_K^{2-\ell-m},
\end{align}
\begin{align}
\psi_2 = \psi_g\psi_h'\norm_K^{2-\ell-m}=\psi_A^{\ell-1}\overline{\psi}_A^{m-1}\norm_K^{2-\ell-m}=\psi_A^{\ell-m}\norm_K^{1-\ell},
\end{align} 
where we have used that $\psi_A'=\overline{\psi_A}$ and $\psi_A\cdot \overline{\psi}_A=\norm_K$.

Let us assume, as in  \S \ref{sec: statement of the conjecture in more general weights}, that we are in a rank $2$ setting. That is to say, $$\dim_L\CH^c(M(f\otimes g \otimes h))_{0,L}=2,$$ say with basis  $ \Delta_1,\Delta_2 $. The main result of this section is Theorem \ref{thm: main general weights} below. It states that assuming Tate's Conjecture for motives (cf. Conjecture \ref{conj: tate}) and a natural property of the $p$-adic Abel--Jacobi map (cf. Assumption \ref{ass: naturality of abel-jacobi}), if $\ord_{s=c}L(f\otimes g\otimes h,s)=2$ and $\L^g_p(\ubff,\ubfg,\ubfh)(k,\ell,m)\neq 0$ then $\operatorname{Reg}(f,g,h)$ is a non-zero algebraic multiple of $\L^g_p(\ubff,\ubfg,\ubfh)(k,\ell,m)$. It can thus be viewed as the (conditional) proof of a particular case of Conjecture \ref{conj: more general weights}.

The strategy of the proof is roughly as follows. Under Tate's Conjecture, the motive $M(f\otimes g\otimes h)$ decomposes as a sum of motives whose underlying varieties are $W_{k-2}\times A^{\ell+m-2}$ and $W_{k-2}\times A^{\ell-m}$. Using this decomposition and Assumption  \ref{ass: naturality of abel-jacobi} we are able to write the regulator $\operatorname{Reg}(f,g,h)$ in terms of cycles in these varieties (Proposition \ref{prop: regulator diagonal}). Then in Proposition \ref{prop: abel jacobi of different cycles} we relate the $p$-adic Abel--Jacobi image of these cycles to that of the generalised Heegner cycles \begin{align}\label{eq: ghcs}
\tilde \Delta^{\psi_i^{-1}}:=\tilde \Delta_{k-2,\ell+m-2,c_i}^{\psi_i^{-1}}\in \CH^{c}(W_{k-2}\times A^{\ell+m-2}/H_{c_i,f})_0,
\end{align}
and then Proposition \ref{prop: main formula in the general case} provides the relation with the special value of the $p$-adic $L$-function.

Before giving the details of this decomposition, as well as the statement and proof of Theorem \ref{thm: main general weights}, we record some basic results on motives attached to Hecke characters and on restriction of scalars of motives.

\subsubsection{Motives attached to certain Hecke characters}\label{sec: motives attached to certain hecke characters} In this subsection we follow \cite[\S 2.2]{BDP4}.  Fix an identification $K\simeq \End(A)$, and for each $\alpha\in K$ let $\alpha^*$ denote the pull back on differentials of the endomorphism of $A$ corresponding to $\alpha$. Recall that $\psi_A$ stands for the Hecke character of $K$ of infinity type $(0,1)$ associated to $A$. 
The motive attached to $\psi_A$ belongs to $\cM(K)_\Q$ and is of the form $$M(\psi_A)=(A,e_{\psi_A},0),$$ for an appropriate projector $e_{\psi_A}$.   The de Rham realization of this motive is the $K$-vector space $$e_{\psi_A}H^1_{\dR}(A).$$ It is endowed with an action $[\cdot ]$ of $K$ given as follows: if $\omega$ is a differential form on $A$ and $\alpha \in K$, then $[\alpha]\omega=\alpha^*\omega$. Fix a holomorphic differential $\omega_A$ on $A/K$ such that $[\alpha]\omega_A = \alpha\omega_A$ for all $\alpha\in K$, and let $\eta_A$ be the unique element of $H^1_\dR(A)$ such that $[\alpha]\eta_A = \overline{\alpha}\eta_A$ for all $\alpha\in K$ and $\langle \omega_A,\eta_A\rangle=1$ (where $\langle \cdot ,\cdot \rangle$ stands for the Poincar\'{e} pairing). The Hodge filtration of $M(\psi_A)_\dR$ is:
\begin{align*}
\Fil^0(M(\psi_A)_\dR) &= K\cdot \omega_A+K\cdot \eta_A ,\\
\Fil^1(M(\psi_A)_\dR) & = K\cdot \omega_A,\\
\Fil^i(M(\psi_A)_\dR) & = 0 \text{ for } i \geq 2.
\end{align*}

Now, for $r\in \Z_{>0}$ consider the motive $M(\psi_A^r)$ associated to $\psi_A^r$. It is of the form 
\begin{align*}
M(\psi_A^r)=(A^r,e_{\psi_A^r},0)
\end{align*}
for a certain projector $e_{\psi_A^r}$ (cf. \cite[\S2.2]{BDP4}). The Hodge filtration is given by:
\begin{align*}
\Fil^0(M(\psi_A^r)_\dR) &= K\cdot \omega_A^r+K\cdot \eta_A^r ,\\
\Fil^i(M(\psi_A^r)_\dR) & = K\cdot \omega_A^r,\text{ for } i=1,\dots,r,\\
\Fil^i(M(\psi_A^r)_\dR) & = 0 \text{ for } i > r.
\end{align*}

\subsubsection{Restriction of scalars of motives}  There is a restriction of scalars functor
\begin{align*}
\Res_{K/\Q}\colon \cM(K)\lra \cM(\Q),
\end{align*}
which extends the restriction of scalars on algebraic varieties to the category of motives (see \cite{karpenko}). 

Suppose that $M\in \cM(K)$, and put $R=\Res_{K/\Q}(M)$. Also, for $Y\in \cM(\Q)$ denote by $Y_K=Y\otimes_\Q K$ the extension of scalars of $Y$ from $\Q$ to $K$. Then there is a canonical morphism $$w:M\ra R_K$$ satisfying the following universal property: if $Y\in \cM(\Q)$ and $f$ is a morphism $f:Y_K\ra M$, then there exists a unique morphism $s:Y\ra R$ such that $w\circ f=s$. In other words, there is a canonical identification
\begin{align*}
\Hom(Y_K,M)\simeq \Hom(Y,\Res_{K/\Q} (M)).
\end{align*}
In particular, 
\begin{align}\label{eq: iso of chow groups res}
\CH^c(M) = \Hom(\mathbb{L}_K^c,M)\simeq \ \Hom(\mathbb{L}^c,\Res_{K/\Q} M)=\CH^c(\Res_{K/\Q} M).
\end{align}
Here we have used that the Lefschetz motive over $K$ is the base extension $\mathbb{L}_K$. We will need the following generalization of \eqref{eq: iso of chow groups res}.
\begin{lemma}\label{lemma: chow groups and res}
	Suppose that $M$ is a motive over $K$ and $N$ is a motive over $\Q$. There is a canonical isomorphism of Chow groups
	\begin{align*}
	\CH^c(N\otimes \Res_{K/\Q}(M))\simeq \CH^c(\norm_K\otimes M).
	\end{align*}
\end{lemma}
\begin{proof}
	By definition of Chow group, and using the standard formula relating tensor products and duals (see \cite[\S1.5]{scholl}) we have:
	\begin{align*}
	\CH^c(N\otimes \Res_{K/\Q}(M))=&\Hom(\mathbb{L}^c, N\otimes \Res_{K/\Q}(M))= \Hom(\mathbb{L}^c\otimes N^\vee, \Res_{K/\Q}(M))\\ =&\Hom(\mathbb{L}_K^c\otimes \norm_K^\vee, M)=\Hom(\mathbb{L}_K^c, \norm_K\otimes M) = \CH^c(\norm_K\otimes M).
	\end{align*}
\end{proof}
There is a natural isomorphism of $\Q$-vector spaces, preserving the Hodge filtration (cf. \cite[p. 16]{jannsen}) 
\begin{align}\label{eq: de rham realization of restriction}
H_{\mathrm{dR}}(M) \simeq H_{\mathrm{dR}}(\Res_{K/\Q}(M)).
\end{align}

We will make extensive use of the well-known fact that the restriction of scalars of a motive and the motive itself have the same $L$-function, that is
\begin{align*}
L(M,s)= L(\Res_{K/\Q}(M),s).
\end{align*}

\subsubsection{A decomposition of $M(f\otimes g\otimes h)$ and the main result}\label{sec: statement of the particular case higher weights}
Recall that the motive over $\Q$ associated to $f$ is $M_f=(W_{k-2},e_f,0)$, and let  $M_{f/K}$ be  its base change to $K$.  As explained in \S\ref{sec: motives attached to certain hecke characters},  we have the following motives in $\cM(K)_\Q$	
\begin{align*}
M(\psi_A^{\ell+m-2})=(A^{\ell+m-2}, e_1, 0), \   M(\psi_A^{\ell-m})=(A^{\ell-m}, e_2,0)
\end{align*}
for suitable projectors that we now denote $e_1$ and $e_2$.  

Define the Hecke characters
\begin{align*}
\tilde{\psi}_1:=\psi_g\psi_h=\psi_A^{\ell+m-2}, \ \ \ \ \tilde{\psi}_2:=\psi_g\psi'_h=\psi_A^{\ell-m}\norm_K^{m-1}.
\end{align*}

For $i=1,2$, denote by $M_i$ the motive associated to $\tilde\psi_i$.  Observe that (see cf. \cite[p. 98]{scha}):
\begin{align*}
M_1=M(\psi_A^{\ell+m-2}),\ \ M_2 = M(\psi_A^{\ell-m})(1-m) \ \ \text{ (the Tate twist)}.
\end{align*}

We need to assume the following classical conjecture.
\begin{conjecture}[Tate's conjecture]\label{conj: tate}
	Let $F$ be a number field and denote by $Rep_{\Q_\ell}(G_F)$ the category of $\ell$-adic representations of $\Gal(\bar F/F)$. The functor 
	$$(\cdot)_\ell:\cM(F)_\Q\longrightarrow  Rep(G_F)$$ that sends a motive over $F$ to its \'{e}tale $\ell$-adic realization is fully faithful.
\end{conjecture}

\begin{proposition}\label{prop: decomposition of chow and derham}
	Assuming Conjecture \ref{conj: tate}, there are natural isomorphisms 
\begin{align}
	\beta_{\CH}\colon   \CH^c(M(f\otimes g\otimes h))_0 \simeq \CH^c \left( M_{f/K}\otimes  M(\psi_A^{\ell+m-2})\right)_0\oplus \CH^{c-m+1} \left( M_{f/K}\otimes  M(\psi_A^{\ell-m})\right)_0
\end{align}
	and
	\begin{align*}
	\beta_\dR\colon   (M(f)_\dR \otimes (M(g\otimes h))_\dR \simeq  ( M_f)_\dR \otimes \left[(M(\psi_A^{\ell+m-2})_\dR \oplus M(\psi_A^{\ell-m})(1-m) )_\dR\right].
	\end{align*}
\end{proposition}
\begin{proof}
	By \eqref{eq the triple product representation splits} and Artin formalism  we have that 
	\begin{align*}
	L(f\otimes g\otimes h,s) =&L(V_f\otimes (V_{\tilde\psi_1}\oplus V_{\tilde\psi_2}),s) =  L(V_f\otimes V_{\tilde\psi_1},s)\cdot L(V_f\otimes V_{\tilde\psi_2},s)\\
	= &L(f/K\otimes \tilde\psi_1,s)\cdot L(f/K\otimes \tilde\psi_2,s)=L(M_f/K\otimes M(\tilde\psi_1))\cdot L(M_f/K\otimes M(\tilde\psi_2))\\ = & L(M_f\otimes \Res_{K/\Q}(M_1))\cdot L(M_f\otimes \Res_{K/\Q}(M_2)).
	\end{align*}
	Tate's conjecture implies then the existence of an isomorphism of motives
	\begin{align*}
	M(f\otimes g\otimes h,s) \simeq M_f\otimes \left( \Res_{K/\Q}(M_1) \oplus \Res_{K/\Q}(M_2)\right),
	\end{align*}
	which induces isomorphisms at the level of Chow groups and de Rham realizations:
	\begin{align*}
	\CH^c(M(f\otimes g\otimes h)) \simeq \CH^c \left( M_f\otimes \left( \Res_{K/\Q}(M_1) \oplus \Res_{K/\Q}(M_2)\right) \right)
	\end{align*}
	\begin{align*}
	\left(M(f\otimes g\otimes h)\right)_\dR \simeq  \left( M_f\otimes \left( \Res_{K/\Q}(M_1) \oplus \Res_{K/\Q}(M_2)\right) \right)_\dR.
	\end{align*}

	By Lemma \ref{lemma: chow groups and res} and using the fact that the cycle class map commutes with the restriction of scalars (see \cite[p.75]{jannsen}), we see that there is a natural isomorphism 
	\begin{align}\label{eq: beta chow}
	\CH^c(M(f\otimes g\otimes h))_0 \simeq \CH^c \left( M_{f/K}\otimes  M_1\right)_0\oplus \CH^c \left( M_{f/K}\otimes  M_2\right)_0.
	\end{align}
	Observe also that there is a canonical isomorphism 
	\begin{align*}
	\CH^c \left( M_{f/K}\otimes  M_2\right)\simeq \CH^{c-m+1} \left( M_{f/K}\otimes  M(\psi_A^{\ell-m})\right).
	\end{align*}
	Indeed, this follows from the very definition of the Chow group of a motive and the fact that $M_2(1-m)=M_2\otimes \mathbb{L}^{m-1}$. Therefore we obtain the canonical isomorphism $\beta_{\CH}$. 
	Also, the isomorphism \eqref{eq: de rham realization of restriction} gives the natural isomorphism $\beta_\dR$.
\end{proof}

Recall that we are in a situation of algebraic rank $2$. We will further assume that we are in a \textit{rank $(1,1)$-setting}, meaning that the rank of both $\CH^c(M_{f/K}\otimes M(\psi_A^{\ell+m-2}))_0$ and $\CH^{c-m+1}(M_{f/K}\otimes M(\psi_A^{\ell-m}))_0$ is one. 
This hypothesis is not too restrictive for the aim of this section. Indeed, the proof Theorem \ref{thm: main general weights} shows that it is satisfied whenever $\L_p^g(\ubff,\ubfg,\ubfh)(k,\ell,m)\neq0$, and we will prove the main result under this non-vanishing hypothesis.

Since in the definition of $\mathrm{Reg}(f,g,h)$ we are free to choose the basis of $\CH^c(M(f\otimes g\otimes h))_0$ we can, and do, assume that $(\Delta_1,\Delta_2)$ are chosen to be adapted to the decomposition of Chow groups given by the isomorphism $\beta_{\CH}$. That is to say, we can suppose that 
\begin{align}\label{eq: beta chow}
\beta_{\CH}(\Delta_1)=(\Delta_1^1, 0), \ \ \beta_{\CH}(\Delta_2)=(0, \Delta_2^2)
\end{align}
for some cycles $\Delta_1^1\in \CH^c(M_{f/K}\otimes M(\psi_A^{\ell+m-2}))_0$ and $\Delta_2^2\in \CH^{c-m+1}(M_{f/K}\otimes M(\psi_A^{\ell-m}))_0$.

In view of the naturalness of the isomorphisms of Proposition \ref{prop: decomposition of chow and derham}, it is also natural to assume that they behave well with respect to the $p$-adic Abel-Jacobi map.

\begin{assumption}\label{ass: naturality of abel-jacobi}
	For any cycle $\Delta$ and de Rham class $\omega$, we have that 
	\begin{align*}
	\AJ_p(\Delta)(\omega)=\AJ_p(\beta_{\CH}(\Delta))(\beta_\dR(\omega)).
	\end{align*}
\end{assumption}

\begin{proposition}\label{prop: regulator diagonal}
	Under the assumptions of this subsection, the matrix defining  $\operatorname{Reg}(f,g,h)$ can be chosen to be diagonal. More precisely,
	\begin{align*}
	\operatorname{Reg}(f,g,h)= \AJ_p(\Delta_1^1)(\eta_A^{\ell+m-2})\cdot \AJ_p(\Delta_2^2)(\eta_A^{\ell-m})\pmod{K^\times}.
	\end{align*} 
\end{proposition}

In order to prove the Proposition, we need a lemma on the behavior of the de Rham classes via the isomorphism $\beta_\dR$.

\begin{lemma}
	If we regard the target of $\beta_\dR$ as the direct sum
	\begin{align*}
	\left(   (M_f)_\dR \otimes M(\psi_A^{\ell+m-2})_\dR \right) \bigoplus \left( (M_f)_\dR \otimes M(\psi_A^{\ell-m})(1-m)_\dR \right),
	\end{align*}
	then we have that:
	\begin{align*} 
	\beta_\dR(\omega_f\wedge\eta_g\wedge\eta_h) = (\omega_f\wedge\eta_A^{\ell+m-2},0) \pmod{K^\times}\\
	\beta_\dR(\omega_f\wedge\eta_g\wedge\omega_h) =(0, \omega_f\wedge \eta_A^{\ell-m})\pmod{K^\times}.
	\end{align*}
\end{lemma}
\begin{proof}
	The isomorphism $\beta_\dR$ is induced from an isomorphism $$M(g\otimes h)_\dR\simeq M(\psi_A^{\ell+m-2})_\dR\oplus M(\psi_A^{\ell-m})(1-m)_\dR$$ that respects the Hodge filtration. Observe that
	\begin{align*}
	\Fil^{0}M(g\otimes h)_\dR/\Fil^{m-1}M(g\otimes h)_\dR=\langle \eta_g\wedge\eta_h\rangle\\
	\Fil^{m-1}M(g\otimes h)_\dR/\Fil^{\ell-1}M(g\otimes h)_\dR=\langle \eta_g\wedge\omega_h\rangle,
	\end{align*}
	On the other hand, we have that:
	\begin{align*}
	\Fil^0 (M(\psi_A^{\ell+m-1}))_\dR/\Fil^{m-1}(M(\psi_A^{\ell+m-1}))_\dR&=\eta_A^{\ell+m-2},\\
	\Fil^{m-1} (M(\psi_A^{\ell+m-1}))_\dR/\Fil^{\ell-1}(M(\psi_A^{\ell+m-1}))_\dR&=0.
	\end{align*}
	As for $(M(\psi_A^{\ell-m})(1-m))_\dR$, recall that it is isomorphic to $(M(\psi_A^{\ell-m}))_\dR$ with the Hodge filtration shifted $(m-1)$-positions. That is:
	\begin{align*}
	\Fil^0 (M(\psi_A^{\ell-m})(1-m))_\dR &=\dots =  \Fil^{m-1} (M(\psi_A^{\ell-m})(1-m))_\dR = \langle \omega_A^{\ell-m},\eta_A^{\ell-m}\rangle \\
	\Fil^m (M(\psi_A^{\ell-m})(1-m))_\dR &= \dots   \Fil^{\ell-1} (M(\psi_A^{\ell-m})(1-m))_\dR =\langle \omega_A^{\ell-m}\rangle\\
	\Fil^\ell (M(\psi_A^{\ell-m})(1-m))_\dR &= 0.
	\end{align*}
	Therefore
	\begin{align*}
	\Fil^{0} (M(\psi_A^{\ell-m})(1-m))_\dR /\Fil^{m-1}(M(\psi_A^{\ell-m})(1-m))_\dR&=0,\\
	\Fil^{m-1} (M(\psi_A^{\ell-m})(1-m))_\dR /\Fil^{\ell-1}(M(\psi_A^{\ell-m})(1-m))_\dR&=\langle\eta_A^{\ell-m}\rangle.
	\end{align*}	
	
\end{proof}

\begin{proof}[Proof of Proposition \ref{prop: regulator diagonal}]
	By Assumption \ref{ass: naturality of abel-jacobi}, the regulator of $f$, $g$, $h$, can be computed as  
	\begin{align*}
	\operatorname{Reg}(f,g,h)=\left|
	\begin{array}{rr}
	\AJ_p(\beta_{\CH}(\Delta_1))(\beta_\dR(\omega_f\wedge \eta_g\wedge\omega_h)) &  \AJ_p(\beta_{\CH}(\Delta_1))(\beta_\dR(\omega_f\wedge \eta_g\wedge\eta_h))\\
	\AJ_p(\beta_{\CH}(\Delta_2))(\beta_\dR(\omega_f\wedge \eta_g\wedge\omega_h)) &  \AJ_p(\beta_{\CH}(\Delta_2))(\beta_\dR(\omega_f\wedge \eta_g\wedge\eta_h))
	\end{array}
	\right|.
	\end{align*}
	
By choosing a basis of the Chow group satisfying \eqref{eq: beta chow}, we find that
\begin{align*}
\operatorname{Reg}(f,g,h)=&\left|
\begin{array}{rr}
\AJ_p((\Delta_1^1,0))((\eta_A^{\ell+m-2},0)) &  \AJ_p((\Delta_1,0))((0,\eta_A^{\ell-m}))\\
\AJ_p((0,\Delta_2^2))((\eta_A^{\ell+m-2},0)) &  \AJ_p((0,\Delta_2^2))((0,\eta_A^{\ell-m}))
\end{array}
\right| \pmod{K^\times}\\ =& \AJ_p(\Delta_1^1)(\eta_A^{\ell+m-2})\cdot \AJ_p(\Delta_2^2)(\eta_A^{\ell-m})\pmod{K^\times}.
\end{align*}

\end{proof}

In order to compare the regulator expressed as in Proposition \ref{prop: regulator diagonal} with the right hand side of (\ref{eq: main formula in the general case}), we focus on the generalised Heegner cycles (\ref{eq: ghcs}) appearing in this setting.

\begin{lemma}
	In the setting of this subsection we have that $H_{c_i,f}=K$. That is to say, the Heegner cycles $\tilde{\Delta}^{\psi_i^{-1}}$ are defined over $K$.
\end{lemma}
\begin{proof}
	Recall that
	$$\psi_1=\psi_A^{m+\ell-2}\norm_K^{2-\ell-m}, \ \ \psi_2=\psi_A^{\ell-m}\norm_K^{1-\ell}.$$
	For $i\in\{ 1,2 \}$, the conductor of $\psi_i$ is of the form $c_i\cdot\cN_{\chi_f}$ where  the norm of $\cN_{\chi_f}$ equals the conductor $N_{\chi_f}$ of $\chi_f$ and $(c_i,N_{\chi_f})=1$.
	
	By (\ref{eq: levels special case}) and (\ref{eq: nebentype chars special case}), the conductor of $\chi_f$ is  $N_{\chi_f}=(D_K)^\epsilon$ where $\epsilon\in\{0,1\}$. On the other hand, the conductor of $\psi_A$ is only divisible by primes above $D_K$, so $c_i=1$. 
	Recall the extension $F/K$ defined in \cite[\S4.2]{BDP2} such that $\Gal(F/K)\cong(\Z/N_{\chi_f}\Z)^\times/(\pm1)$. The field $H_{c_1,f}=H_{c_2,f}=H_{1,f}$ is the subextension of $F/K$ corresponding to $\ker(\chi_f)(\pm1)\subseteq (\Z/N_f\Z)^\times/(\pm1)$, which is $K$ by (\ref{eq: nebentype chars special case}). 
\end{proof}

	Denote $\tilde{\Delta}^{\psi_2^{-1}\norm_K^{1-m}}:=\tilde{\Delta}^{\psi_2^{-1}\norm_K^{1-m}}_{k-1,\ell-m,1}\in\CH^{c-m+1}(W_{k-1}\times A^{\ell-m}/K)_0$.

\begin{proposition}\label{prop: abel jacobi of different cycles}
	\begin{align*}
	\AJ_p(\tilde{\Delta}^{\psi_2^{-1}})(\omega_f\wedge\omega_A^{m-1}\eta_A^{\ell-1})=(2\sqrt{-D_K})^{1-m}\AJ_p(\tilde{\Delta}^{\psi_2^{-1}\norm_K^{1-m}})(\omega_f\wedge\eta_A^{\ell-m}).
	\end{align*}
	
\end{proposition}

\begin{proof}

	By \cite[Proposition 4.1.1]{BDP2}, there is a correspondence
	$$P:W_{k-2}\times A^{\ell+m-2}\longrightarrow W_{k-2}\times A^{\ell-m}. $$
	induced by the cycle
	$$Z=W_{k-2}\times A^{\ell-m}\times A^{m-1}\in\CH^{k+\ell-2}(W_{k-2}\times A^{\ell+m-2}\times W_{k-2}\times A^{\ell-m})$$
	embedded in $$W_{k-2}\times A^{\ell+m-2}\times W_{k-2}\times A^{\ell-m}=W_{k-2}\times A^{\ell-m}\times (A\times A)^{m-1}\times W_{k-2}\times A^{\ell-m}$$ via $$\Id\times(\sqrt{-D_K}\times\Id)^{m-1}\times\Id.$$

	It induces a homomorphism of Chow groups
	
	\[ P_*:\CH^\frac{k+\ell+m-2}{2}(W_{k-2}\times A^{\ell+m-2})\longrightarrow \CH^\frac{k+\ell-m}{2}(W_{k-2}\times A^{\ell-m})      \]

	For each $\fa$ ideal of $\cO_K$ prime to $\cN$, let $$\Delta_{k-1,\fa}\in\CH^{k-1}(W_{k-2}\times A^{k-2}),$$ be the generalised Heegner cycle defined in \cite{BDP1}. As we recalled in \S\ref{sec: Generalised Heegner cycles}, for each $b\leq k-2$ such that $b\equiv k \mod 2$ there is a cycle 
	$$\Delta_{k-1,b,\fa}\in\CH^\frac{k+b}{2}(W_{k-2}\times A^b),$$ as defined in \cite{BDP2}.
	
	The correspondence above gives the relations between these cycles:
	
	$$P_{*}(\Delta_{k-1,\ell+m-2,\fa})=(N\fa)^{m-1}\Delta_{k-1,\ell-m,\fa}.$$

	Using this relation, as in \cite[Proposition 4.1.2]{BDP2}, we obtain
	$$(2\sqrt{-D_K})^{m-1}\AJ_p(\Delta_{k-1,\ell+m-2,\fa})(\omega_f\wedge\omega_A^{m-1}\eta_A^{\ell-1})=(N\fa)^{m-1}\AJ_p(\Delta_{k-1,\ell-m,\fa})(\omega_f\wedge\eta_A^{\ell-m}).$$
	
	Then, finally,

	\begin{equation*}
	\AJ_p(\tilde{\Delta}^{\psi_2^{-1}})(\omega_f\wedge\omega_A^{m-1}\eta_A^{\ell-1})=(2\sqrt{-D_K})^{1-m}\AJ_p(\tilde{\Delta}^{\psi_2^{-1}\norm_K^{1-m}})(\omega_f\wedge\eta_A^{\ell-m}).
	\end{equation*}

\end{proof}

Finally, we state and prove the main result of this section. Recall that $f$, $g$, and $h$ are modular forms of weights $k$, $\ell$, and $m$ respectively with $\ell>m\geq 2$ and $k\geq \ell+m$. In addition, $g$ and $h$ are theta series of an imaginary quadratic field $K$ of class number $1$ that satisfies Assumption \ref{ass: condition on DK} and in which $p$ splits. More precisely, $g=\theta(\psi_A^{\ell-1})$ and $h=\theta(\psi_A^{m-1})$, where $A/K$ is an elliptic curve of conductor $\sqrt{-D_K}$ which has CM by $\cO_K$.

\begin{theorem}\label{thm: main general weights}
	Let $(\ubff,\ubfg,\ubfh)$ be the choice of test vector of Theorem \ref{thm: Hsieh interpolation formula}. Assume that $$ \dim_L\CH^c(M(f\otimes g\otimes h))_{0,L}=2 \text{ and  }\L^g_p(\ubff,\ubfg,\ubfh)(k,\ell,m)\neq 0.$$ Under Conjecture \ref{conj: tate} and Assumption \ref{ass: naturality of abel-jacobi},  there exists a quadratic extension $L_0$ of $L$ and a $\lambda\in L_0$ such that 
	\begin{equation*}
	\L^g_p(\ubff,\ubfg,\ubfh)(k,\ell,m)= \mathrm{Reg}(f,g,h) \mod (K\cdot L_0)^\times.
	\end{equation*} 
\end{theorem}

\begin{proof}
Assume that $\L_p^g(\bff,\bfg,\bfh)(k,\ell,m)\neq0$. Combining Proposition \ref{prop: main formula in the general case} and Proposition \ref{prop: abel jacobi of different cycles} with the fact that the kernel of the $p$-adic Abel--Jacobi map contains all torsion cycles, we obtain that the generalised Heegner cycles $\tilde{\Delta}^{\psi_1^{-1}}$ and $\tilde{\Delta}^{\psi_2^{-1}\norm_K^{1-m}}$ are nontorsion. Since we are in a situation of algebraic rank $2$, this implies that the preimages via $\beta_{\CH}$ of  $$(\tilde{\Delta}^{\psi_1^{-1}},0), \ \ (0,\tilde{\Delta}^{\psi_2^{-1}\norm_K^{1-m}})$$  generate $\CH^c(M(f\otimes g\otimes h))_0$. In other words, we can choose $\Delta_1, \Delta_2$ in such a way that $$\Delta_1^1=\tilde{\Delta}^{\psi_1^{-1}},  \ \ \ \ \Delta_2^2=\tilde{\Delta}^{\psi_2^{-1}\norm_K^{1-m}}.$$ 

	On the other hand, the period $\Omega$ attached to the elliptic curve $A/K$ coincides with the period $\Omega(\psi_A)$ attached to the Hecke character $\psi_A$ as in \cite[\S2.3]{BDP2}. 
	It follows from \cite[Proposition 2.11 (2)]{BDP2} that,  $\Omega(\psi_A^r)=\Omega^r \mod K^\times$ for $r\geq0$. Using  \cite[Proposition 2.11 (2)]{BDP2}, we conclude that the  factor $\mu$ appearing in (\ref{eq: main formula in the general case}) lies in $K^\times$.

	The result then follows by combining Proposition \ref{prop: regulator diagonal}, Proposition \ref{prop: abel jacobi of different cycles} and Proposition \ref{prop: main formula in the general case}.

\end{proof}

\bibliographystyle{halpha}
\bibliography{refs}

\end{document}